\documentclass[11pt]{article}
 
\usepackage[hypertex]{hyperref}

\usepackage{epsfig}
\usepackage{graphicx}

\usepackage[a4paper]{geometry}  
\usepackage{amsmath,amssymb}
\usepackage[latin1]{inputenc}

\usepackage{color}

\usepackage{amssymb}
\usepackage{amscd}
\usepackage{amsmath}
\usepackage{amsfonts}
\usepackage{amsthm}
\usepackage{mathrsfs}

\newtheorem{theo}{Theorem}[section]
\newtheorem{prop}{Proposition}[section]
\newtheorem{lemma}{Lemma}[section]
\newtheorem{cor}{Corollary}[section]

\theoremstyle{remark}
\newtheorem{rem}{Remark} [section]

\newcommand{\R}{{\mathbb{R}}}
\newcommand{\N}{{\mathbb{N}}}
\newcommand{\C}{{\mathbb{C}}}
\newcommand{\Z}{{\mathbb{Z}}}

\newcommand{\tore}{{\mathbb{T}}}
\newcommand{\La}{\Lambda} 
\newcommand{\la}{\lambda} 
\newcommand{\al}{\alpha} 
\newcommand{\be}{\beta} 
 
\newcommand{\te}{\theta} 
\newcommand{\ga}{\gamma}
\newcommand{\Ga}{\Gamma}
\newcommand{\om}{\omega} 
\newcommand{\Si}{\Sigma} 
\newcommand{\ta}{\tau} 
 
\newcommand{\De}{\Delta}
\newcommand{\si}{\sigma} 

\newcommand{\Ci}{{\mathcal{C}}^{\infty}}

\newcommand{\Lie}{{\mathfrak{g}}} 

\newcommand{\Ad}{\operatorname{Ad}} 
\newcommand{\id}{\operatorname{id}} 
\newcommand{\preq}{{\mathcal{P}}}  

\newcommand{\module}{{\mathcal{M}}}   
\newcommand{\planar}{{\mathcal{P}}}
\newcommand{\mreg}{{\mathcal{M}} ^{\operatorname{reg}}}  

\newcommand{\Hom}{\operatorname{Hom}}

\newcommand{\dLie}{{\mathcal{L}}} 

\newcommand{\Hilbert}{{\mathcal{H}}}  

\newcommand{\Ol}{{\mathcal{O}}}   
\newcommand{\Proj}{{\mathbb{P}}}   

\newcommand{\Sp}{\operatorname{Sp}}   

\newcommand{\CL}{\operatorname{CG}} 

\newcommand{\Int}{\operatorname{Int}}  

\newcommand{\dcan}{\wedge ^{\operatorname{top},0} } 
\newcommand{\dvol}{\wedge ^{\operatorname{top}} } 

\newcommand{\Fourier}{{\mathcal{F}}} 

\newcommand{\varparallel}{/\!\! /}  

\newcommand{\End}{\operatorname{End}}  

\newcommand{\Ei}{E_{\operatorname{int}}}
\newcommand{\Eh}{E_{\operatorname{half}}}

\title{On the quantization of polygon spaces}
\author{L. Charles\footnote{Universit{\'e} Pierre et Marie Curie, UMR 7586 Institut de
  Math{\'e}matiques de Jussieu, Paris, F-75005 France.}}

\begin{document}
\maketitle

\begin{footnotesize}
  \noindent \textbf{MS Classification :}
  47L80, 
  53D30, 
  53D12, 
  53D50, 
  53D20, 
  81S10, 
81S30, 
81R12, 
81Q20, 
\end{footnotesize}

\begin{abstract}
Moduli spaces of polygons have been studied since the nineties for
their topological and symplectic properties. Under generic assumptions,
these are symplectic manifolds with natural global action-angle
coordinates. This paper is concerned with the quantization of these manifolds
and of their action coordinates. 
Applying the geometric quantization procedure, one is lead to consider
invariant subspaces of a tensor product of irreducible
representations of $SU(2)$. These quantum spaces admit natural sets of
commuting observables. We prove that these operators form a
semi-classical integrable system, in the sense that they are Toeplitz operators
with principal symbol the square of the action coordinates.
As a consequence, the quantum spaces admit bases whose vectors
concentrate on the Lagrangian submanifolds of constant action. The
coefficients of the change of basis matrices can be estimated in terms of
geometric quantities.  We recover this way the already known asymptotics of the
classical $6j$-symbols. 
\end{abstract}

\bibliographystyle{plain}

\section{Introduction} 

Given an $n$-tuple $\ell = (\ell_1,
\ldots, \ell_n)$ of positive
numbers, the polygon space  $\module_\ell$ consists of the spatial
$n$-sided polygons with side lengths equal to the
$\ell_i$, up to isometries. When $\ell$ satisfies a generic
assumption, $\module_\ell$ is a compact K{\"a}hler  manifold. These moduli
spaces have been studied in
several papers since the nineties for their topology and symplectic
properties  \cite{KaMi}, \cite{HaKn}, \cite{Kl} and \cite{FaHaSc}. Among other
things, Kapovich and Millson discovered in \cite{KaMi} a remarkable action-angle
coordinate system.  Considering
the triangulation in figure \ref{fig:action_angle}, the actions
 are defined as the lengths of the internal edges and the
angles as the dihedral angles between the faces adjacent to the
internal edges. More generally, given any decomposition of a $n$-sided
polygon into triangles obtained by connecting the vertices, one defines
a canonical action-angle coordinate system. 

\begin{figure}
\begin{center}
\begin{picture}(0,0)%
\includegraphics{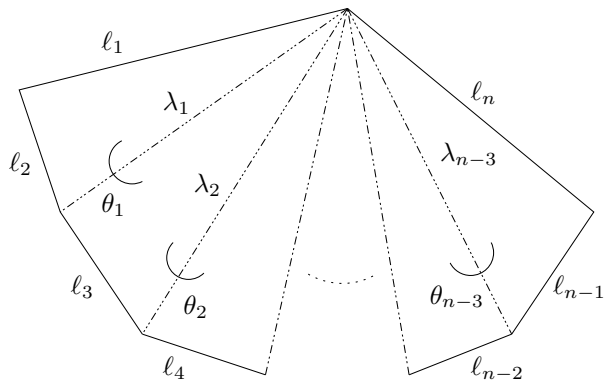}%
\end{picture}%
\setlength{\unitlength}{3355sp}%
\begingroup\makeatletter\ifx\SetFigFont\undefined%
\gdef\SetFigFont#1#2#3#4#5{%
  \reset@font\fontsize{#1}{#2pt}%
  \fontfamily{#3}\fontseries{#4}\fontshape{#5}%
  \selectfont}%
\fi\endgroup%
\begin{picture}(4287,2770)(526,-2219)
\put(1651,-211){\makebox(0,0)[lb]{\smash{{\SetFigFont{10}{12.0}{\rmdefault}{\mddefault}{\updefault}{\color[rgb]{0,0,0}$\la_1$}%
}}}}
\put(1201,239){\makebox(0,0)[lb]{\smash{{\SetFigFont{10}{12.0}{\rmdefault}{\mddefault}{\updefault}{\color[rgb]{0,0,0}$\ell_1$}%
}}}}
\put(1201,-961){\makebox(0,0)[lb]{\smash{{\SetFigFont{10}{12.0}{\rmdefault}{\mddefault}{\updefault}{\color[rgb]{0,0,0}$\te_1$}%
}}}}
\put(1801,-1711){\makebox(0,0)[lb]{\smash{{\SetFigFont{10}{12.0}{\rmdefault}{\mddefault}{\updefault}{\color[rgb]{0,0,0}$\te_2$}%
}}}}
\put(1876,-811){\makebox(0,0)[lb]{\smash{{\SetFigFont{10}{12.0}{\rmdefault}{\mddefault}{\updefault}{\color[rgb]{0,0,0}$\la_2$}%
}}}}
\put(3601,-1636){\makebox(0,0)[lb]{\smash{{\SetFigFont{10}{12.0}{\rmdefault}{\mddefault}{\updefault}{\color[rgb]{0,0,0}$\te_{n-3}$}%
}}}}
\put(976,-1561){\makebox(0,0)[lb]{\smash{{\SetFigFont{10}{12.0}{\rmdefault}{\mddefault}{\updefault}{\color[rgb]{0,0,0}$\ell_3$}%
}}}}
\put(3676,-586){\makebox(0,0)[lb]{\smash{{\SetFigFont{10}{12.0}{\rmdefault}{\mddefault}{\updefault}{\color[rgb]{0,0,0}$\la_{n-3}$}%
}}}}
\put(3901,-2161){\makebox(0,0)[lb]{\smash{{\SetFigFont{10}{12.0}{\rmdefault}{\mddefault}{\updefault}{\color[rgb]{0,0,0}$\ell_{n-2}$}%
}}}}
\put(1651,-2161){\makebox(0,0)[lb]{\smash{{\SetFigFont{10}{12.0}{\rmdefault}{\mddefault}{\updefault}{\color[rgb]{0,0,0}$\ell_4$}%
}}}}
\put(3901,-136){\makebox(0,0)[lb]{\smash{{\SetFigFont{10}{12.0}{\rmdefault}{\mddefault}{\updefault}{\color[rgb]{0,0,0}$\ell_n$}%
}}}}
\put(4501,-1561){\makebox(0,0)[lb]{\smash{{\SetFigFont{10}{12.0}{\rmdefault}{\mddefault}{\updefault}{\color[rgb]{0,0,0}$\ell_{n-1}$}%
}}}}
\put(526,-661){\makebox(0,0)[lb]{\smash{{\SetFigFont{10}{12.0}{\rmdefault}{\mddefault}{\updefault}{\color[rgb]{0,0,0}$\ell_2$}%
}}}}
\end{picture}%
\caption{Action-angle coordinates $(\la_i, \te_i)_{i =1, \ldots , n-3}$.} \label{fig:action_angle}
\end{center}
\end{figure}

The main subject of this paper is the quantum counterpart of these
coordinate systems. Assuming that the lengths $\ell_i$ are integral,
one may apply the geometric quantization procedure to the polygon
space $\module_\ell$. The quantum space is defined as the space of
holomorphic sections of a prequantization line bundle with base
$\module_\ell$. Let $V_m$ be the $(m+1)$-dimensional irreducible representation
of $SU(2)$ and consider the invariant subspace $\Hilbert_\ell$ of the tensor product of the $V_{\ell_i}$:
$$   \Hilbert _\ell := \bigl( V_{\ell_1 } \otimes V_{\ell_2}
\otimes \ldots \otimes V_{\ell_n} \bigr)^{SU(2)}.$$
Then the quantum space associated to $\module_\ell$ is isomorphic to
$\Hilbert_\ell$. If we replace the prequantum bundle by its $k$-th
power and twist it by a half-form bundle, we obtain a quantum space
isomorphic to $\Hilbert _{k\ell-1}$. 
The semi-classical limit is defined as the limit 
$k \rightarrow \infty$, the parameter $k$ corresponding to the inverse
of the Planck constant. 

The usual tools of microlocal analysis
have been introduced in the context of compact K{\"a}hler manifolds and
may be applied to the quantization of the polygon spaces. In
particular there exists a class of operators, called Toeplitz
operators, which plays a role similar to the class of
pseudo-differential operators with small parameter
(\cite{BoGu}, \cite{oim_bt}). Also relevant to this paper are the Lagrangian
sections (\cite{BPU}, \cite{oim_qm}). These are families of sections which
in the semiclassical limit concentrate on a Lagrangian submanifold in a
precise way. They are similar
to the usual Lagrangian distributions of microlocal analysis.  

Let us return to the vector space $\Hilbert_\ell$. Since the
Littlewood-Richardson coefficients of $SU(2)$ are $0$ or $1$, to each
bracketing of the product $V_{\ell_1} \otimes \ldots \otimes
V_{\ell_n}$ corresponds a decomposition of $\Hilbert_{\ell}$ into a direct sum of
lines. By using concrete irreducible
representations, we can even introduce well-defined bases of $\Hilbert_\ell$. 
In particular, the
classical $6j$-symbols are defined as the coefficients of the change of
basis matrix between basis coming from $((V_{\ell_1} \otimes V_{\ell_2}) \otimes V_{\ell_3})
  \otimes V_{\ell_4}$ and $(V_{\ell_1} \otimes
(V_{\ell_2} \otimes V_{\ell_3}))
  \otimes V_{\ell_4}.$  

For each way of placing  bracket, we will define a family $(H_{i,
  \ell})_{i =1, \ldots , n-3}$ of mutually commuting operators of
$\Hilbert _\ell$ whose joint eigenspaces are the summand of the
associated decomposition. 
Our main result says that the sequences $(H_{i,k\ell
  -1})_k$ are Toeplitz operators of $\module_\ell$, cf. theorem \ref{theo:main}. Further, their principal
symbols are the squares of the actions of a canonical
coordinate system defined as above. For instance, the parenthesising
$$ ((\ldots ((V_{\ell_1} \otimes V_{\ell_2} ) \otimes V_{\ell_3} )
\ldots ) \otimes V_{\ell_{n-1}}) \otimes V_{\ell_n}$$
  corresponds to the angle coordinates defined in figure \ref{fig:action_angle}. 

Families of
  pseudodifferential operators whose principal symbols form an
  integrable system have been the subject of many works, cf. \cite{Sa}
  and references therein. Their joint spectrum can be computed in
  terms of the geometry of the integrable system by the
  Bohr-Sommerfeld conditions and their joint eigenstates are Lagrangian
  functions. These results have been extended to semi-classical
  integrable systems of Toeplitz operators in \cite{oim_qm} and
  \cite{oim_hf} and can therefore be applied to the operators $(H_{i,k\ell
  -1})_k$. In this case, the Bohr-Sommerfeld conditions are not so interesting
because their joint spectrum can easily be computed explicitely. On the other hand, it is a non-trivial result that the joint
eigenspaces are generated by Lagrangian sections associated to the Lagrangian
submanifolds of constant action. We prove this and also compute the symbol
of these Lagrangian sections, cf. theorem \ref{theo:quasi_mode}.  

As a consequence, we deduce the surprising
formula of Roberts \cite{Ro} which relates the asymptotics of $6j$-symbols to
the geometry of the tetrahedron, cf. theorem
\ref{theo:asymptotism-6j}. As a matter of fact,  the scalar product of two Lagrangian sections can be estimated in
terms of their symbols when the associated Lagrangian
manifolds intersect transversally. 

Our main motivation to study the
quantization of polygon spaces, besides being a natural application
of our previous works, is the similarity with topological quantum
field theory (TQFT). In this case we consider instead of $\module_\ell$ the
moduli space of flat $SU(2)$-principal bundles on a surface $\Si$ with
prescribed holonomy on the boundary. This
moduli space admits natural Lagrangian fibrations associated to
each decomposition of $\Si$ into pairs of pants \cite{JeWe}. The
quantum space is the space of conformal blocks of $\Si$ and
its dimension is given by the famous Verlinde formula.
The manifestation of the Lagrangian fibration at the quantum level has
been considered in several papers. 
Jeffrey and
Weitsmann showed in \cite{JeWe} that the number of fibres satisfying a
Bohr-Sommerfeld condition is given by the Verlinde
formula. Taylor and Woodward conjectured in \cite{WoTa} that some bases of
the quantum space of the four-holed two-sphere consist of Lagrangian
sections. They deduced the asymptotics of the quantum
$6j$-symbols. Their heuristic argument was an inspiration
for our estimation of the classical $6j$-symbol.

 Also, in the TQFT context, some operators called multicurve
operators, play an important role.  In particular, 
the operators associated to the curves of
a decomposition into pair of pants pairwise commute and have one-dimensional joint
eigenspaces. It has been conjectured in \cite{An} that the multicurve operators are
Toeplitz operators with principal symbol a holonomy function. Several
pieces of evidence support this: the symbols of the product and
of the commutators are given at first order by the usual product and the
Poisson bracket \cite{Tu}. Furthermore the trace of the operator is equivalent
in the semiclassical limit to the average of its symbol \cite{MaNa}. 
We hope that our method could also apply to this context or at least
that our results could clarify what we can expect.

The article is organized as follows. Section \ref{sec:space-hilbert_ell} is devoted to algebraic
preliminaries. We introduce the
decompositions of $\Hilbert_\ell$ as a direct sum of lines and the
associated famillies of mutually commuting operators. Next section
concerns the geometric quantization of the polygon spaces. Applying
the "quantization commutes with reduction" theorem of Guillemin and
Sternberg \cite{GuSt}, we prove that it is isomorphic to  $\Hilbert_\ell$. We also state
the central result that the sequences $(H_{i, k\ell -1})$ are
Toeplitz operators. The main part of the proof is postponed to the
last section of the paper. Section \ref{sec:sympl-geom-polyg} is
devoted to the symplectic geometry of the polygon spaces. Generalizing
the result of \cite{KaMi}, we associate to any triangle decomposition
an action-angle
coordinate system. We also describe carefully the image of
the action coordinates and the associated torus actions. In section \ref{sec:semi-class-prop}, we deduce that the joint eigenstates of the $(H_{i,
  k\ell -1})$ are Lagrangian sections. We also state the
Bohr-Sommerfeld conditions. 
Section \ref{sec:scal-prod-lagr} is
concerned with the asymptotics of the scalar product of two Lagrangian
sections. At first order, it is given by a geometric pairing between
the symbols of the Lagrangian sections. This is then
applied in section \ref{sec:asympt-6j-symb} to estimate the classical
$6j$-symbols. The main part of the proof was already understood by
Woodward and Taylor \cite{WoTa}, except for the delicate phase
determination. 
In the last section we
consider the symplectic reduction of Toeplitz operators. We compute the principal and subprincipal
symbols of a reduced Toeplitz operator. These subprincipal estimates
are the most difficult results of the paper.

{\bf Acknowledgment} 
I would like to thank Julien March{\'e} for many helpful discussions and
his interest in this work.

\section{On the space $\Hilbert_{\ell}$} \label{sec:space-hilbert_ell}

\subsection{Outline of the following sections}

In section \ref{sec:notations},  we introduce some notations and a class of
graphs called admissible graphs. 
To each admissible graph will be
associated a set of action coordinates of the polygon space
$\module_{\ell}$ and also a set of mutually commuting operators of the
Hilbert space $\Hilbert_{\ell}$. So these graphs will be used to relate the choice of a parenthesising of the
product $V_{\ell_1} \otimes \ldots \otimes V_{\ell_n}$ with the choice
of a decomposition of a polygon into triangles. In the analogy with
the moduli space of connections and the topological quantum field
theory, the admissible graphs correspond to the graphs associated to the
decomposition into pairs of pants of a surface.

In section \ref{sec:decomp} we explain how each parenthesising of $V_{\ell_1} \otimes
\ldots \otimes V_{\ell_n}$ leads to a decomposition of
$\Hilbert_{\ell}$ as a direct sum of lines. 
In the next section we define for each admissible
graph a set of mutually commuting operators of
$\Hilbert_{\ell}$. Their joint spectrum is explicitly computed and
their joint eigenspaces are the lines of the decomposition
associated to a particular parenthesising. In the last section we show
a first relationship between the spaces $\Hilbert_{\ell}$ and the polygons,
namely that the existence of a non-vanishing vector in
$\Hilbert_{\ell}$ is equivalent to a parity condition and the
existence of a $n$-sided polygon with side lengths $\ell_1, \ldots,
\ell_n$. 
\subsection{Notations} \label{sec:notations}

Our normalisation for the invariant scalar product of $\mathfrak{su}(2)$
is 
$$ \langle  \xi , \eta \rangle = - \frac{1}{2} \operatorname{tr} ( \xi \eta),$$
where we identify $\mathfrak{su}(2)$ with the space of skew-Hermitian
endomorphisms of $\C^2$. 
As in the introduction, for any non-negative integer $m$, we denote by 
$V_m$ the irreducible representation
of $SU(2)$ with spin $m/2$, i.e. the dimension of $V_m$ is $m+1$. For any
$n$-tuple $\ell = (\ell_i)_{i =1, \ldots , n}$ of non-negative
integers, $\Hilbert_\ell$ is the vector space 
\begin{gather} \label{eq:def_Hilb}
  \Hilbert _{\ell}  := \operatorname{Inv}\bigl( V_{\ell_1} \otimes V_{\ell_2}
\otimes \ldots \otimes V_{\ell_n } \bigr) .
\end{gather}

Let us introduce some notations regarding graphs. We say that a graph
$\Ga$ is {\em admissible} if it is connected, acyclic and trivalent. Edges with
only one endpoint are permitted and called {\em half-edges}. The other edges
have two endpoints and are called {\em internal edges}. We denote by $\Ei (\Ga)$ the set of
internal edges, by $\Eh (\Ga)$ the set of
half-edges and by $E( \Ga) = \Ei (\Ga) \cup \Eh (\Ga)$ the set of edges.
We always
assume that the half-edges are numbered. Abusing notations, we often
identify $\Eh (\Ga)$ with  $\{ 1, \ldots , n\}$, where $n$ is the
number of half-edges.

A {\em coloring} of
$\Ga$  is an assignment of a non-negative
integer to each edge of $\Ga$. We say that three non-negative integers
$m$, $\ell$ and $p$ satisfy the {\em Clebsch-Gordan} condition $\CL (p, \ell,
m)$ if  the following holds
\begin{gather}  \label{eq:Clebsch_Gordan_condition}
p + \ell + m \in 2 \Z , \quad p \leqslant \ell + m , \quad \ell \leqslant m + p , \quad m
\leqslant \ell + p
\end{gather}
A coloring of $\Ga$ is called {\em admissible} if for
any vertex $t$, the colors of the edges incident to $t$ satisfy the Clebsch-Gordan condition.

\subsection{Natural decompositions of $\Hilbert_\ell$} \label{sec:decomp}

Let us recall that the multiplicity spaces of $SU (2)$ are
one or zero-dimensional according to the Clebsch-Gordan condition
(\ref{eq:Clebsch_Gordan}), 
$$ \dim \Hom _ {SU(2)} (V_k , V_\ell \otimes V_m ) = \begin{cases} 1 \text{ if }
  \CL ( k , \ell, m ) \text{ holds } \\ 0 \text{ otherwise}
\end{cases} $$
So the decomposition into isotypical spaces of the tensor product of
two irreducible representation is given by 
\begin{xalignat}{2} \label{eq:Clebsch_Gordan}
 V_\ell \otimes V_m = & \bigoplus_k \Hom _ {SU(2)} (V_k , V_\ell \otimes
V_m ) \otimes V_k 
=  \bigoplus_{k \; / CG ( k, \ell , m) } V_k 
\end{xalignat}
Given a parenthesising of the tensor product of $n$ irreducible
representations, one may decompose the full product into irreducible
representations by applying (\ref{eq:Clebsch_Gordan}) $(n-1)$ times in
the order prescribed by the brackets. As instance
\begin{xalignat*}{2} 
  &  V_{\ell_1}  \otimes  (V_{\ell_2} \otimes ( V_{\ell_3} \otimes  V_{\ell_4} ))  =  \bigoplus _{   m_1 /  \CL (m_1,
   \ell_3, \ell_4 ) }    V_{\ell_1}  \otimes  (V_{\ell_2} \otimes
   V_{m_1} )  \\
= & \bigoplus _{ \substack{ m_1, m_2 \; / \; \CL (m_1, \ell_3 , \ell_4)
    \\ \& \CL (m_2, \ell_2, m_1 )  }}   V_{\ell_1} \otimes  V_{m_2}  =
\bigoplus _{ \substack{ m_1, m_2, m_3 \;  / \; \CL (m_1, \ell_3 , \ell_4)  \\  \CL (m_2,
   \ell_2, m_1 ) \; \& \; \CL (m_3, \ell_1, m_2)}}    V_{m_3} 
\end{xalignat*}
Then taking the invariant part we obtain a decomposition into a direct
sum of lines. In the previous example, we obtain
\begin{xalignat*}{2} 
    (V_{\ell_1}  \otimes  (V_{\ell_2} \otimes ( V_{\ell_3} \otimes
  V_{\ell_4} )))^{SU(2)}  & = \bigoplus _{ \substack{ m_1, m_2 \;  / \; \CL (m_1, \ell_3 , \ell_4)  \\  \CL (m_2,
   \ell_2, m_1 ) \; \& \; \CL (0, \ell_1, m_2)}}    V_{0} \\
 & = \bigoplus _{ \substack{ m_1 \;  / \; \CL (m_1, \ell_3 , \ell_4)  \\  \CL (\ell_1,
   \ell_2, m_1 ) }}    V_{0}  
\end{xalignat*}
because $\CL (0, \ell_1,
m_2)$ is equivalent to $\ell_1 = m_2$.
In general let us show that the lines of
the decomposition are naturally indexed by admissible colorings of a
certain admissible graph.  First one associates to each parenthesising of a
product of $n$ factors,  a tree with $n$ leaves, $n-2$ trivalent vertices and one
bivalent vertex (called the root), cf. figure \ref{fig:parenthesage}
for a few examples.
\begin{figure}
\begin{center}
\begin{picture}(0,0)%
\includegraphics{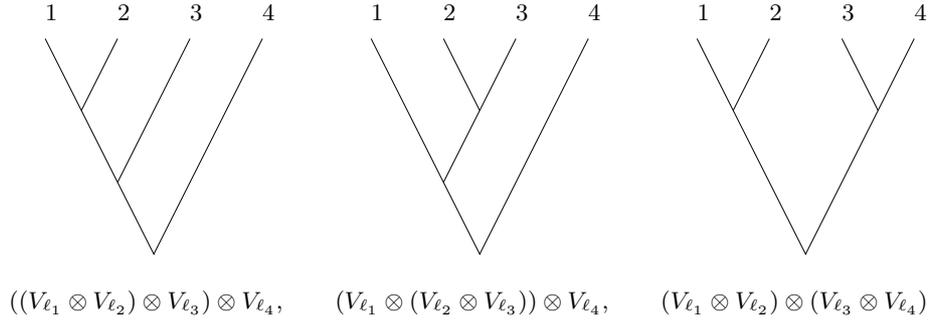}%
\end{picture}%
\setlength{\unitlength}{2960sp}%
\begingroup\makeatletter\ifx\SetFigFont\undefined%
\gdef\SetFigFont#1#2#3#4#5{%
  \reset@font\fontsize{#1}{#2pt}%
  \fontfamily{#3}\fontseries{#4}\fontshape{#5}%
  \selectfont}%
\fi\endgroup%
\begin{picture}(7776,2590)(301,-4169)
\put(601,-1711){\makebox(0,0)[lb]{\smash{{\SetFigFont{9}{10.8}{\rmdefault}{\mddefault}{\updefault}{\color[rgb]{0,0,0}$1$}%
}}}}
\put(1201,-1711){\makebox(0,0)[lb]{\smash{{\SetFigFont{9}{10.8}{\rmdefault}{\mddefault}{\updefault}{\color[rgb]{0,0,0}$2$}%
}}}}
\put(1801,-1711){\makebox(0,0)[lb]{\smash{{\SetFigFont{9}{10.8}{\rmdefault}{\mddefault}{\updefault}{\color[rgb]{0,0,0}$3$}%
}}}}
\put(2401,-1711){\makebox(0,0)[lb]{\smash{{\SetFigFont{9}{10.8}{\rmdefault}{\mddefault}{\updefault}{\color[rgb]{0,0,0}$4$}%
}}}}
\put(3301,-1711){\makebox(0,0)[lb]{\smash{{\SetFigFont{9}{10.8}{\rmdefault}{\mddefault}{\updefault}{\color[rgb]{0,0,0}$1$}%
}}}}
\put(3901,-1711){\makebox(0,0)[lb]{\smash{{\SetFigFont{9}{10.8}{\rmdefault}{\mddefault}{\updefault}{\color[rgb]{0,0,0}$2$}%
}}}}
\put(4501,-1711){\makebox(0,0)[lb]{\smash{{\SetFigFont{9}{10.8}{\rmdefault}{\mddefault}{\updefault}{\color[rgb]{0,0,0}$3$}%
}}}}
\put(5101,-1711){\makebox(0,0)[lb]{\smash{{\SetFigFont{9}{10.8}{\rmdefault}{\mddefault}{\updefault}{\color[rgb]{0,0,0}$4$}%
}}}}
\put(6001,-1711){\makebox(0,0)[lb]{\smash{{\SetFigFont{9}{10.8}{\rmdefault}{\mddefault}{\updefault}{\color[rgb]{0,0,0}$1$}%
}}}}
\put(6601,-1711){\makebox(0,0)[lb]{\smash{{\SetFigFont{9}{10.8}{\rmdefault}{\mddefault}{\updefault}{\color[rgb]{0,0,0}$2$}%
}}}}
\put(7201,-1711){\makebox(0,0)[lb]{\smash{{\SetFigFont{9}{10.8}{\rmdefault}{\mddefault}{\updefault}{\color[rgb]{0,0,0}$3$}%
}}}}
\put(7801,-1711){\makebox(0,0)[lb]{\smash{{\SetFigFont{9}{10.8}{\rmdefault}{\mddefault}{\updefault}{\color[rgb]{0,0,0}$4$}%
}}}}
\put(301,-4111){\makebox(0,0)[lb]{\smash{{\SetFigFont{9}{10.8}{\rmdefault}{\mddefault}{\updefault}{\color[rgb]{0,0,0}$((V_{\ell_1} \otimes V_{\ell_2}) \otimes V_{\ell_3}) \otimes V_{\ell_4}$,}%
}}}}
\put(3001,-4111){\makebox(0,0)[lb]{\smash{{\SetFigFont{9}{10.8}{\rmdefault}{\mddefault}{\updefault}{\color[rgb]{0,0,0}$(V_{\ell_1} \otimes (V_{\ell_2} \otimes V_{\ell_3})) \otimes V_{\ell_4}$,}%
}}}}
\put(5701,-4111){\makebox(0,0)[lb]{\smash{{\SetFigFont{9}{10.8}{\rmdefault}{\mddefault}{\updefault}{\color[rgb]{0,0,0}$(V_{\ell_1} \otimes V_{\ell_2}) \otimes (V_{\ell_3} \otimes V_{\ell_4})$}%
}}}}
\end{picture}%
\caption{The graphs associated to various parenthesisings} \label{fig:parenthesage}
\end{center}
\end{figure}

A coloring of such a tree is the assignment of an integer to
each edge. It is admissible if the colors of the edges incident
to the root are equal and if furthermore at each trivalent vertex, the colors of
the incident edges satisfy the Clebsch-Gordan condition. Now given a
parenthesising and the associated tree, it is clear that the summands
of the decomposition coming from the parenthesising are indexed by the
admissible colorings  of the tree such that the edge incident to the
$k$-th leave is colored by $\ell_k$. 

Finally given a tree as above, let us consider the admissible graph
obtained by removing the leaves and the root and by merging the two edges
incident to the root, cf. as instance figure \ref{fig:tree_graph}. 
\begin{figure}
\begin{center}
\begin{picture}(0,0)%
\includegraphics{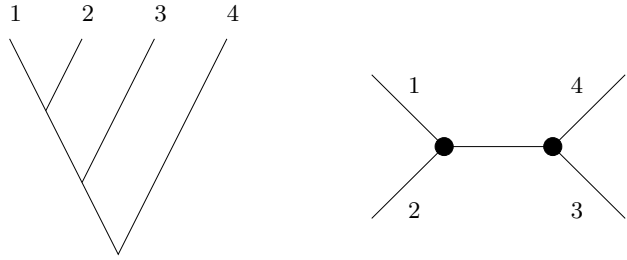}%
\end{picture}%
\setlength{\unitlength}{2960sp}%
\begingroup\makeatletter\ifx\SetFigFont\undefined%
\gdef\SetFigFont#1#2#3#4#5{%
  \reset@font\fontsize{#1}{#2pt}%
  \fontfamily{#3}\fontseries{#4}\fontshape{#5}%
  \selectfont}%
\fi\endgroup%
\begin{picture}(5124,2094)(589,-3673)
\put(601,-1711){\makebox(0,0)[lb]{\smash{{\SetFigFont{9}{10.8}{\rmdefault}{\mddefault}{\updefault}{\color[rgb]{0,0,0}$1$}%
}}}}
\put(1201,-1711){\makebox(0,0)[lb]{\smash{{\SetFigFont{9}{10.8}{\rmdefault}{\mddefault}{\updefault}{\color[rgb]{0,0,0}$2$}%
}}}}
\put(1801,-1711){\makebox(0,0)[lb]{\smash{{\SetFigFont{9}{10.8}{\rmdefault}{\mddefault}{\updefault}{\color[rgb]{0,0,0}$3$}%
}}}}
\put(2401,-1711){\makebox(0,0)[lb]{\smash{{\SetFigFont{9}{10.8}{\rmdefault}{\mddefault}{\updefault}{\color[rgb]{0,0,0}$4$}%
}}}}
\put(3901,-2311){\makebox(0,0)[lb]{\smash{{\SetFigFont{9}{10.8}{\rmdefault}{\mddefault}{\updefault}{\color[rgb]{0,0,0}$1$}%
}}}}
\put(3901,-3361){\makebox(0,0)[lb]{\smash{{\SetFigFont{9}{10.8}{\rmdefault}{\mddefault}{\updefault}{\color[rgb]{0,0,0}$2$}%
}}}}
\put(5251,-2311){\makebox(0,0)[lb]{\smash{{\SetFigFont{9}{10.8}{\rmdefault}{\mddefault}{\updefault}{\color[rgb]{0,0,0}$4$}%
}}}}
\put(5251,-3361){\makebox(0,0)[lb]{\smash{{\SetFigFont{9}{10.8}{\rmdefault}{\mddefault}{\updefault}{\color[rgb]{0,0,0}$3$}%
}}}}
\end{picture}%
\caption{A tree and its associated admissible graph} \label{fig:tree_graph}
\end{center}
\end{figure}
Observe that the admissible colorings of the
graph are in one-to-one correspondence with the admissible colorings of
the tree. To summarize we have proven the following

\begin{prop} \label{sec:decomp-graphe}
Let $n \geqslant 3$ and $\ell$ be a $n$-tuple of
non-negative integers.  
For any way of parenthesizing the product $V_{\ell_1} \otimes \ldots
\otimes V_{\ell_n}$, there exists an admissible graph $\Ga$ with $n$
half-edges and a
decomposition into a sum of lines    
$$\Hilbert_{\ell} = \bigoplus D_{\varphi} $$
where $\varphi$ runs over the colorings of $\Ga$ such that $\varphi ( i
) = \ell_i$ for any half-edge $i$. 
\end{prop}

Observe that different parenthesisings may be associated to the same graph. As we will see in the next section not only the indexing set but also
the summands $D_\varphi$ of the decomposition depend only on the graph $\Ga$.

\subsection{A complete set of observables} \label{sec:definition_Ha}

Recall first that the Casimir operator $Q$ of a representation $\rho$
of $SU(2)$ on a finite dimensional vector space $V$ is the
self-adjoint operator $V \rightarrow V$  defined by
\begin{gather} \label{eq:Casimir2}
 Q = - (  \rho ( \xi_1 ) ^2 +  \rho ( \xi_2 ) ^2 +  \rho ( \xi_3 ) ^2)
\end{gather}
where $(\xi_1, \xi_2 , \xi_3)$ is an orthogonal base of $\mathfrak{su}
(2)$. Its eigenspaces are the isotypical
subspaces of $V$, the eigenvalue $ 4 j ( j+ 1)$ corresponding to the spin $j$ isotypical
subspace.

The following lemma will be useful in the sequel. Consider a representation $\rho_1 \times \rho_2$  of $SU (2) \times
SU(2)$ on a finite dimensional vector space $V$.  Denote by $\rho$ the
$SU(2)$-representation on $V$ by the diagonal inclusion $\rho(g) = \rho_1(g)
\times \rho_2 (g)$. 

\begin{lemma} \label{sec:lemme_stupide}
The Casimir operators $Q_1$ and $Q_2$ of the representations $\rho_1$ and
$\rho_2$ commute with $\rho$. Furthermore the restrictions
of $Q_1$ and $Q_2$  to the $\rho$-invariant part of $V$ coincide. 
\end{lemma} 

\begin{proof} 
To check this, it suffices to decompose $V$ as the sum of isotypical subspaces
for the product representation $\rho_1 \times \rho_2$. Indeed $Q_1$ and $Q_2$
preserve this decomposition and act
by multiplication by $a(a+2)$ and $b(b+2)$ respectively on the $V_a
\otimes V_b$-isotypical subspace. Furthermore,  it follows from
(\ref{eq:Clebsch_Gordan})  that the
$\rho$-invariant part of the $V_a
\otimes V_b$-isotypical subspace is trivial if $a \neq b$. 
\end{proof} 
 For any subset $I$ of $\{1, \ldots , n \}$, let $\rho_{I,\ell}$ be the representation 
of $SU(2)$ in the tensor product  $V_{\ell_1} 
\otimes \ldots \otimes V_{\ell_n }$ given by  
\begin{gather*} 
   \rho_{I, \ell}(g)  ( v_{1} \otimes \ldots \otimes
v_{n} )  =  w_1 \otimes \ldots \otimes w_n , \qquad \text{ where } w_i
= \begin{cases} g.v_i \text{ if } i \in I \\ v_i \text{ otherwise.} 
\end{cases}
\end{gather*}
In particular $\rho_{\{1, \ldots , n\}, \ell} $ is the diagonal
representation whose invariant subspace is $\Hilbert_{\ell}$. 
By the previous lemma, the Casimir operator $Q_{I, \ell}$ of $\rho_{I,
  \ell}$ preserves the subspace
$\Hilbert_{\ell}$. Furthermore the restrictions of $ Q_{I, \ell}$ and $Q_{
  I^c, \ell}$ to $\Hilbert_{\ell}$ coincide.

Consider now an admissible graph $\Ga$ with $n$ half-edges. 
Let $a$ be an internal edge of $\Ga$. Assume $a$ is directed and
consider the set $I(a)$ of half-edges
which are connected to the initial vertex of $a$ by a path which
does not contain the terminal vertex of $a$. 
 We define the operator 
\begin{gather} \label{eq:def_Ha}
H_{a, \ell}: \Hilbert _{\ell} \rightarrow \Hilbert_{\ell}
\end{gather}
as the restriction of $Q_{I(a), \ell}$ to
${\Hilbert_\ell}$. If we change the direction of $a$, we replace
$I(a)$ by its complementary subset. So
the operator $H_{a, \ell}$ is defined independently of the orientation of the
edge. 

\begin{theo} \label{sec:spectre} Let $n \geqslant 4$. Let $\Ga$ be an
  admissible graph with $n$
  half-edges. Let $\ell$ be a $n$-tuple of non-negative integers. Then the operators $H_{a, \ell}$, $a \in \Ei(\Ga)$, mutually commute. The joint
eigenvalues of these operators are the $(n-3)$-tuples 
$$ \bigl( {\varphi}(a)({\varphi}(a) +2)
\bigr)_{a\in \Ei (\Ga) } $$
where $\varphi$ runs over the admissible colorings of $\Ga$ such that
$\varphi (i ) = \ell_i$ for any half-edge $i$. Each joint
eigenvalue is simple.  

Furthermore if $\Ga$ is the graph associated to
a parenthesising of $V_{\ell_1} \otimes \ldots \otimes V_{\ell_n}$ as
in proposition \ref{sec:decomp-graphe}, the joint eigenspace associated to the coloring
$\varphi$ is the line $D_{\varphi}$ of proposition \ref{sec:decomp-graphe}. 
\end{theo}

\begin{proof}
To prove that $H_{a, \ell}$ and $H_{b, \ell}$ commute for any internal edges $a$
and $b$, we use that these operators do not depend on the direction of the edges. Changing these directions if necessary, we
may assume that $I(a)$ and $I(b)$ are disjoint. Then the
representations $\rho_{I(a), \ell}$ and $\rho_{I(b), \ell}$
commute. So their Casimir operators also commute.

 Observe that for any permutation $\si$ of $\{ 1, \ldots , n \}$, the canonical
isomorphism 
$$ V_{\ell_1} \otimes \ldots \otimes V_{\ell_n} \rightarrow
V_{\ell_{\si(1)}} \otimes \ldots \otimes V_{\ell_{\si(n)}}$$
intertwins $Q_{I(a), \ell}$ with $Q_{\si(I(a)), \si(\ell)}$ where $\si
  (\ell) = (\ell_{\si (1)}, \ldots, \ell_{\si (n)})$. Hence it is
  sufficient to prove the first part of the theorem for one order of the
  half-edges. 
It is easily seen that the graph $\Ga$ endowed with an appropriate order of the
half-edges is induced by a parenthesising of the product $V_{\ell_1}
\otimes \ldots \otimes V_{\ell_n}$. So it suffices to show that
each summand $D_\varphi$ of the associated decomposition is the joint
eigenspace of the $(H_{a, \ell})$ with joint eigenvalue
$(\varphi(a)( \varphi (a) + 2))$. 

To see this, one considers first the tree associated to the
parenthesising and orient each edge from the leaves to the
root. Then one may associate a representation $\rho_{I(a), \ell}$ to
each internal edge $a$ of the tree exactly as we did for
admissible graphs. Recall that to
define the summands $D_\varphi$, we first express $V_{\ell_1} \otimes
\ldots \otimes V_{\ell_n}$ as 
a direct sum of irreducible representations by decomposing the partial
products with (\ref{eq:Clebsch_Gordan}). This amounts to decompose $V_{\ell_1} \otimes \ldots \otimes V_{\ell_n}$
with respect to the isotypical subspaces of the $\rho_{I(a), \ell}$ and
$\rho_{\{1, \ldots, n \}, \ell}$, the last one for the
root. Equivalently one considers the direct sum of the joint eigenspaces of the
associated Casimir operators $Q_{I(a), \ell}$ and $Q_{\{1, \ldots, n
  \}, \ell }$. Then we take the invariant part which is the kernel of 
$Q_{\{1, \ldots, n \}, \ell }$. To conclude observe that the operators
$Q_{I(a), \ell}$ associated to the internal edges of the tree are the
same as the ones associated to the internal edges of $\Ga$. 
\end{proof}

\subsection{On the non-triviality of $\Hilbert_{\ell}$}

Since it is preferable to know whether $\Hilbert_\ell$ is trivial, we prove the
following 

\begin{prop} \label{cor:le_cor}
Let $n \geqslant 1$ and $\ell$ be a $n$-tuple of
non-negative integers.  
Then $\Hilbert_{\ell}$ is not trivial if and only if
the sum $\ell_1 +
\ldots + \ell_n$ is even and there exists a $n$-sided planar polygon with
side lengths $\ell_1$, \ldots, $\ell_n$. 
\end{prop}

Furthermore the existence of a planar polygon with side lengths
$\ell_1, \ldots, \ell_n$ is
equivalent to the inequalities $\ell_j \leqslant \tfrac{1}{2} (\ell_1 +
\ldots + \ell_n )$ for $j=1, \ldots , n$.  
\begin{proof} 
If $n=1,2, 3$,  the result is easily proved. For $n \geqslant 3$,
consider the graph $\Ga_{n}$ in figure \ref{fig:gamma_n}. 
\begin{figure}
\begin{center}
\begin{picture}(0,0)%
\includegraphics{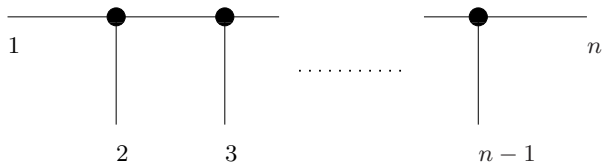}%
\end{picture}%
\setlength{\unitlength}{2960sp}%
\begingroup\makeatletter\ifx\SetFigFont\undefined%
\gdef\SetFigFont#1#2#3#4#5{%
  \reset@font\fontsize{#1}{#2pt}%
  \fontfamily{#3}\fontseries{#4}\fontshape{#5}%
  \selectfont}%
\fi\endgroup%
\begin{picture}(5088,1332)(1189,-3110)
\put(6001,-2161){\makebox(0,0)[lb]{\smash{{\SetFigFont{9}{10.8}{\rmdefault}{\mddefault}{\updefault}$n$}}}}
\put(5101,-3061){\makebox(0,0)[lb]{\smash{{\SetFigFont{9}{10.8}{\rmdefault}{\mddefault}{\updefault}$n-1$}}}}
\put(3001,-3061){\makebox(0,0)[lb]{\smash{{\SetFigFont{9}{10.8}{\rmdefault}{\mddefault}{\updefault}{\color[rgb]{0,0,0}$3$}%
}}}}
\put(1201,-2161){\makebox(0,0)[lb]{\smash{{\SetFigFont{9}{10.8}{\rmdefault}{\mddefault}{\updefault}{\color[rgb]{0,0,0}$1$}%
}}}}
\put(2101,-3061){\makebox(0,0)[lb]{\smash{{\SetFigFont{9}{10.8}{\rmdefault}{\mddefault}{\updefault}{\color[rgb]{0,0,0}$2$}%
}}}}
\end{picture}%
\caption{The graph $\Ga_n$} \label{fig:gamma_n}
\end{center}
\end{figure}
By proposition \ref{sec:decomp-graphe}, the non-triviality of
$\Hilbert_\ell$ is equivalent to the existence of a coloring $\varphi$
of $\Ga_n$ such that $\varphi (i ) = \ell_i$ for any $i$. If such a
coloring exists, by summing the parity conditions
(\ref{eq:Clebsch_Gordan_condition}) at each vertex we obtain that the
sum of the $\ell_i$ is even. Furthermore the inequalities in
(\ref{eq:Clebsch_Gordan_condition}) are equivalent to the existence of
a triangle with side lengths $k$, $\ell$, $m$. Given a coloring
of $\Ga$, one can patch together
the triangles associated to the various vertices as in figure
\ref{fig:pol_gamma_n}. One obtains a
$n$-sided polygon with the required side lengths. 

\begin{figure}
\begin{center}
\begin{picture}(0,0)%
\includegraphics{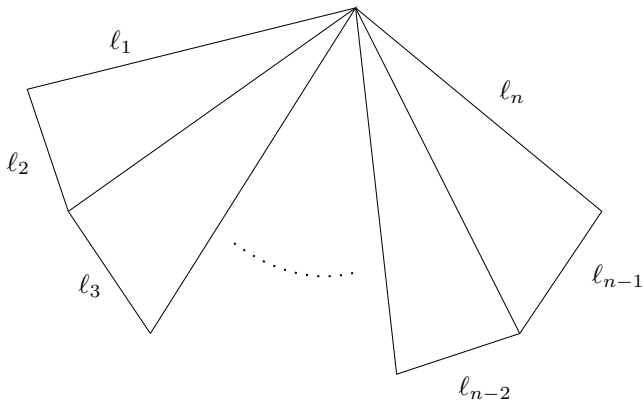}%
\end{picture}%
\setlength{\unitlength}{3355sp}%
\begingroup\makeatletter\ifx\SetFigFont\undefined%
\gdef\SetFigFont#1#2#3#4#5{%
  \reset@font\fontsize{#1}{#2pt}%
  \fontfamily{#3}\fontseries{#4}\fontshape{#5}%
  \selectfont}%
\fi\endgroup%
\begin{picture}(4362,2920)(451,-2369)
\put(451,-661){\makebox(0,0)[lb]{\smash{{\SetFigFont{10}{12.0}{\rmdefault}{\mddefault}{\updefault}{\color[rgb]{0,0,0}$\ell_2$}%
}}}}
\put(1201,239){\makebox(0,0)[lb]{\smash{{\SetFigFont{10}{12.0}{\rmdefault}{\mddefault}{\updefault}{\color[rgb]{0,0,0}$\ell_1$}%
}}}}
\put(4726,-1486){\makebox(0,0)[lb]{\smash{{\SetFigFont{10}{12.0}{\rmdefault}{\mddefault}{\updefault}{\color[rgb]{0,0,0}$\ell_{n-1}$}%
}}}}
\put(4051,-136){\makebox(0,0)[lb]{\smash{{\SetFigFont{10}{12.0}{\rmdefault}{\mddefault}{\updefault}{\color[rgb]{0,0,0}$\ell_n$}%
}}}}
\put(3751,-2311){\makebox(0,0)[lb]{\smash{{\SetFigFont{10}{12.0}{\rmdefault}{\mddefault}{\updefault}{\color[rgb]{0,0,0}$\ell_{n-2}$}%
}}}}
\put(976,-1561){\makebox(0,0)[lb]{\smash{{\SetFigFont{10}{12.0}{\rmdefault}{\mddefault}{\updefault}{\color[rgb]{0,0,0}$\ell_3$}%
}}}}
\end{picture}%
\caption{The triangulation associated to $\Ga_n$} \label{fig:pol_gamma_n}
\end{center}
\end{figure}

For the converse, we may assume without restriction that the lengths
$\ell_i$ do not vanish. Then the proof is by induction on $n$. For $n \geqslant 3$,
consider a $(n+1)$-tuple $(\ell_i)$ satisfying the parity condition and
assume
that there exists a family of vectors $(v_i)$ with lengths
$|v_i|=\ell_i$ and such that the sum $v_1+ \ldots + v_{n+1}$
vanishes. Then it suffices to prove that we can choose these vectors in
such a way that $$\ell'_n  = |v_n + v_{n+1} |$$ is integral and
$\ell'_n + \ell_n + \ell_{n+1}$ is even. Indeed, if it is the case,
one may apply the induction assumption to $(\ell_1, \ldots ,
\ell_{n-1}, \ell_n')$ and we get a admissible coloring $\varphi$ of
the graph $\Ga_n$ with $\varphi(i) = \ell_i$ for $i=1, \ldots , n-1$
and $\varphi (n) = \ell_n'$. Then we extend this coloring to an admissible
coloring of $\Ga_{n+1}$ by assigning $\ell_n$ and $\ell_{n+1}$ to the
$n$-th and $(n+1)$-th half-edges.

Let us prove the existence of the $v_i$ satisfying the extra condition. The minimum and
maximum values of $|v_n +v_{n+1}|$ when $v_n$ and $v_{n+1}$
run over the circles of radius $\ell_n$ and $\ell_{n+1}$ are   
$$ m = |\ell_n - \ell_{n+1}|, \qquad M = \ell_n +
\ell_{n+1} .$$
Similarly if $|v_i | =\ell_i$ for $i =1, \ldots , n-1$, the minimum and maximum of $|v_1 + \ldots + v_{n-1}|$ are
$$ m ' = \max  (2 L - ( \ell_1 + \ldots + \ell_{n-1}), 0 )  , \qquad M'= \ell_1 + \ldots +\ell_{n-1} $$
where $L$ is the maximum of $\ell_1, \ldots , \ell_{n-1}$.
Observe also that  $m < M $ and $m ' < M '$   because the
lengths $\ell_i$ are positive. 

Then the existence of the family $(v_i)$ such that $v_1 + \ldots +
v_{n+1} = 0$ is equivalent to the existence of a length $\ell'_n \in
[m , M ] \cap [m ' , M ' ]$. We have to prove that if this intersection
is non-empty then it contains an integer $\ell_n'$ such that $\ell'_n
+ \ell_n + \ell_{n+1}$ is even. Since the endpoints $m$, $M$,
$m'$ and $M'$ are all integral, if $[m , M ] \cap [m '
, M ' ]$ is not empty, it contains at least two consecutive integers
or it consists of one point. In this last case, one has $M =m'$ or $m
= M'$, hence $\ell'_n = m$ or $M$ which is integral and satisfies the parity condition.  
\end{proof}

\section{Semi-classical properties of $\Hilbert_{\ell}$}

Applying the geometric quantization procedure, we give a
geometric construction of the spaces $\Hilbert_{\ell}$. We recall
the general setting in a first section and explain in the next
one how the
spaces $\Hilbert_{\ell}$ are viewed as quantization of polygon
spaces. In the following sections, we define Toeplitz operators and
then prove that the operators $H_{a, \ell}$ introduced previously are
Toeplitz operators.

\subsection{Geometric quantization} \label{sec:gq}
Consider a compact connected symplectic
manifold $(M, \om)$. Assume that it is endowed with: 
\begin{itemize} 
\item 
a K{\"a}hler
structure whose fundamental form is $\om$,
\item a prequantization bundle
$L \rightarrow M$ , that is a holomorphic Hermitian line bundle whose
Chern curvature is $ \frac{1}{i} \om$, 
\item  a
half-form bundle $\delta \rightarrow M$, that is a square root of the
canonical bundle of $M$.
\end{itemize}
The quantum space associated to these data
is the space of holomorphic sections of $L \otimes \delta$ that we
denote by $H^0 (M,
L\otimes \delta)$. It has the scalar product 
$$ (\Psi_1, \Psi_2) = \int_M ( \Psi_1(x) , \Psi_2(x) )_{L_x \otimes
  \delta_x} \; \mu_M (x)$$ 
induced by the metric of $L \otimes \delta$ and the Liouville measure  $\mu_M = \om^n / n!$. 

Consider a Hamiltonian action of a Lie group $G$ on $M$ with momentum  $\mu : M
\rightarrow \Lie^*$. By definition, the momentum is equivariant and
satisfies  
 $$ d \mu^\xi + \om ( \xi ^ \# , . ) = 0, \qquad \forall \xi \in \Lie
 $$
 where $\xi ^\#$ is the infinitesimal vector field of $M$ associated
 to $\xi$. We assume that the action lifts to the prequantization
 bundle  in such a way that the infinitesimal action on the sections
 of $L$ is 
\begin{gather} \label{eq:lift_action}
 \nabla_{\xi^\#}^L + i \mu^\xi 
\quad \forall \xi \in \Lie .
\end{gather}
Furthermore we assume that the action on $M$ preserves the complex
structure and lifts to the half-form bundle. With these assumptions,
the group $G$ acts naturally on $H^0 (M, L \otimes \delta)$, the
infinitesimal action being given by the Kostant-Souriau operators
\begin{gather}  \label{eq:KS}    i \mu^\xi +   \nabla_{ \xi
  ^\# }^L \otimes \id + \id \otimes \dLie_{\xi^\#}  , \qquad
\forall \xi \in \Lie
\end{gather}
Here $\dLie$ is the Lie derivative of half-forms. 

All the irreducible representations of compact Lie groups may be obtained in
this way. Starting from the irreducible representations  of
$SU(2)$, we will construct the tensor product $V_{\ell_1} \otimes \ldots
\otimes V_{\ell_n}$ and its invariant subspace as quantum spaces
associated to some symplectic manifolds.

\subsection{Geometric realization of the space $\Hilbert_{\ell}$}
\label{sec:quant_geom}

Consider the tautological bundle $\Ol
(-1)$ of the projective complex line $\Proj^1 ( \C)$ and let $\Ol ( m) = \Ol (-1) ^{- m}$. These are $SU(2)$-bundles with base $\Proj^1( \C)$. For
 any integer $m \geqslant 1$, the induced $SU(2)$-action on
the  holomorphic sections of  
$$  \Ol (m) \otimes \Ol (-1) \rightarrow \Proj^1 ( \C)$$ 
is the irreducible representation $V_{m-1}$.
 Denote by $\om_{FS}$ the Fubiny-Study form of $\Proj^1 (\C)$. Then
 $\Ol (m)$ is a prequantization bundle with curvature $ \frac{1}{i}
 m \om_{FS}$. Furthermore $\Ol (-1)$ is the unique half-form bundle
 of $ \Proj^{1} ( \C)$. The $SU(2)$-action satisfies all the general
 properties stated below.  Its momentum 
$$\mu_m :
\Proj ^1(\C)  \rightarrow (\mathfrak{su} (2) )^* $$
is an embedding whose image is the coadjoint orbit with symplectic
volume  $2\pi m$. With our normalization of the scalar product, this coadjoint orbit is the
sphere $S^2_m$ with radius $m$ centered at the origin.  
In the following we identify $(\Proj ^1 ( \C), m \om_{FS})$
with $S^2_m$.

Next we consider the product $S^2_{\ell_1} \times \ldots \times
S^2_{\ell_n}$ with the prequantization bundle and the half-form bundle:
\begin{gather} \label{eq:fibres}
  \Ol ( \ell_1) \boxtimes \ldots \boxtimes \Ol ( \ell_n),  \qquad   \Ol ( -1) \boxtimes \ldots \boxtimes \Ol ( -1)
\end{gather}
The associated quantum space is the tensor product $V_{\ell_1-1}
\otimes \ldots \otimes V_{\ell_n-1}$. 
The diagonal action of $SU(2)$ is Hamiltonian with momentum 
$$ \mu = \pi_1^* \mu_{\ell_1} + \pi_2^* \mu_{\ell_2} +  \ldots +
\pi_n^* \mu_{\ell_n} ,$$
where $\pi_i$ is the projection from $S_{\ell_1} \times \ldots \times
S_{\ell_n}$ onto the $i$-th factor. 
It satisfies all the previous assumptions. 

The invariant subspace of $V_{\ell_1 -1} \otimes \ldots \otimes
V_{\ell_n -1}$  is the Hilbert space $\Hilbert _{\ell -1}$ we
introduced previously. It follows from the
"quantization commutes with reduction" theorem proved in \cite{GuSt}
that $\Hilbert _{ \ell -1}$ is the quantum space associated with the
symplectic quotient $\module_\ell$ of $S^2_{\ell_1} \times \ldots \times
S^2_{\ell_n}$ by  $SU (2)$. More precisely we assume that  
\begin{gather} \label{eq:hyp_existence}
\ell_j \leqslant \tfrac{1}{2} (\ell_1 +
\ldots + \ell_n ) , \qquad j=1, \ldots , n.
\end{gather}
and
\begin{gather} \label{eq:hyp2}
  \ell_1 \pm \ell_2 \pm \ldots \pm \ell_n \neq 0.
\end{gather}
for any possible choice of signs. The first assumption is equivalent
to the non-emptyness of $\module_\ell$. The second one is to
ensure that the quotient does not have any singularity. Indeed,
(\ref{eq:hyp2}) holds if and only if the  $SU(2)$-action on the null
set 
$\{ \mu =0 \} $ factorizes through a free $SO(3)$-action. So under
this assumption the
quotient $\module_{\ell}$ is a manifold. It inherits by reduction a K{\"a}hler
structure. We assume furthermore that  
\begin{gather}\label{eq:hyp1}
 \ell_1 + \ldots + \ell_n \in 2\Z \quad \text{ and } \quad n \in 2 \Z 
\end{gather}
Then the diagonal $SU(2)$-action on the line bundles
(\ref{eq:fibres}) factors through a $SO(3)$-action. The quotients of the restriction at $\{ \mu
=0 \}$ of these bundles are genuine line bundles with base $\module
_\ell $, that we denote by
$L_{\ell}$ and $\delta_{\ell}$ respectively. $L_{\ell}$ has a
natural structure of prequantization bundle and $\delta_{\ell}$ is
a half-form bundle. The restrictions at  $\{ \mu
=0 \}$ of the invariant sections descend to the quotient $\module_\ell$. This
defines a map 
 $$\Hilbert_{\ell -1} \rightarrow H^0 ( \module_\ell , L_{\ell}
\otimes  \delta_{\ell}) $$
which is a vector space isomorphism. Additional details on this
construction 
will be recalled in section \ref{sec:data_reduction}. Applying the
same construction to the power of the prequantization bundle, we obtain the following theorem. 
\begin{theo} \label{theo:geom-real-space}
Let $n \geqslant 4$ and $\ell = (\ell_1, \ldots ,
  \ell_n)$ be a family of positive integers satisfying
  (\ref{eq:hyp_existence}), (\ref{eq:hyp2}) and (\ref{eq:hyp1}). Then for any positive integer $k$, we have a
  natural vector space isomorphism 
$$ 
V_{k, \ell} :   \Hilbert_{k \ell -1 } \rightarrow H^0 (\module_\ell , L_{\ell}^k \otimes
\delta_{\ell} ). $$
\end{theo}

Since the maps $V_{k,\ell}$ do not necessarily preserve the scalar
product, we will also consider the unitary operators $V_{k, \ell} (V_{k, \ell}^* V_{k,
  \ell})^{-1/2}$. The asymptotic results we prove in this paper are 
in the limit $k \rightarrow \infty$.
\begin{rem}
In the case where assumption (\ref{eq:hyp2}) does not hold, one can
still identify the Hilbert space $\Hilbert_{k \ell -1}$ with a space of
holomorphic sections on a K{\"a}hler analytic space \cite{Sj}. We will not consider
this case because the theory of Toeplitz operators has not been developed for singular spaces. 
The assumption (\ref{eq:hyp1}) is not
really necessary. Actually we could extend our results to the sequences 
$$ \Hilbert_{k
  \ell -1 + m }, \qquad k=1,2, \ldots $$ 
where $\ell$ and $m$ are any multi-indices satisfying (\ref{eq:hyp2}). Recall that $\Hilbert_{\varphi}
$ is a trivial vector space when  $ |\varphi | = \varphi (1) + \ldots +\varphi
 ( n) $ is odd. Assume that $|\ell|$ and $n$ are even. then we may
 consider without restriction that $|m|$ is even too. Let $K
 \rightarrow \module_\ell$ be the
 quotient bundle of the restriction at $\{ \mu =0 \}$ of $ \Ol ( m_1)
\boxtimes \ldots \boxtimes \Ol ( m_n)$. We have an isomorphism 
$$ \Hilbert_{k \ell -1 +
  m } \simeq  H^0 (\module_\ell , L_{\ell}^k \otimes
K \otimes
\delta_{\ell} ).$$
Then we can apply the methods used in this paper as it is
explained in \cite{oim_hf} and \cite{oim_mc}. 
When $|\ell|$ is even whereas $n$ and $|m|$ are odd, we can not define
globally the bundles $K$ and
$\delta_{\ell} $. But their tensor product is still perfectly
defined. When $|\ell|$ is odd, by doing the parameter change $k = 2 k'
+1$ or $k = 2k'$, we are reduced to the previous cases.   
Here to avoid the complications due to the auxiliary
bundle $K$, we consider uniquely the spaces $\Hilbert_{k
  \ell -1}$  
under the assumption that $|\ell|$ and $n$ are even. \qed
\end{rem}

\subsection{Toeplitz operators} \label{sec:toeplitz-operators}

Let $(M, \om)$ be a compact K{\"a}hler manifold with a prequantization
bundle $L$ and a half-form bundle $\delta$. For any integer $k$ and
any function $f \in
\Ci (M)$, we consider the rescaled Kostant-Souriau operator (cf. (\ref{eq:KS}))
$$ \preq_k (f) = f + \frac{1}{ik } \Bigl( \nabla^{L^k}_X \otimes \id +
\id \otimes \dLie_X^\delta \Bigr) : \Ci (M, L^k \otimes \delta) \rightarrow \Ci (M, L^k \otimes \delta) 
$$ 
where $X$ is the Hamiltonian vector field of $f$. 
These operators do not necessarily preserve the subspace of
holomorphic sections. Let us consider $H^0 (M ,  L^k \otimes \delta )$
as a subspace of the space of $L^2$-sections. Assume this last space
is  endowed with the scalar product induced by the
metrics of $L$, $\delta$ and the Liouville form. Introduce the
orthogonal projector $\Pi_k$ onto  $ H^0 (M, L^k \otimes \delta
)$. Then a Toeplitz operator is defined as a family of operators  
$$ \Bigl( T_k := \Pi_k \preq_k (f(\cdot, k)) + R_k : H^0 (M, L^k \otimes \delta
) \rightarrow H^0 (M, L^k \otimes \delta ) \Bigr)_{ k =1,2, \ldots}  $$ 
where $(f(\cdot, k))$ is a sequence of $\Ci (M)$ which admits an
asymptotic expansion 
$$ f(\cdot, k ) = f_0 + k^{-1} f_1 + k^{-2} f_2 + \ldots 
$$
for the $\Ci$-topology with $f_0$, $f_1 \ldots \in \Ci
(M)$. Furthermore $(R_k)$ is any operator such that  $\| R_k \| =
O(k^{-\infty})$. As a result the coefficients $f_\ell$ are
determined by $(T_k)$. We call $f_0$ the {\em principal} symbol of $(T_k)$
and $f_1$ its {\em subprincipal} symbol. 

\subsection{The Casimir operators} 
Consider now the symplectic quotient $\module _\ell$ and its
quantization as in section \ref{sec:quant_geom}. For any subset $I$ of $\{ 1, \ldots, n \}$, we
denote by $q_{I, \ell}$ the function 
\begin{gather} \label{eq:symbole_casimir}
 q_{I, \ell} := \Bigl| \sum_{i \in I} \pi_i^* \mu_{\ell_i} \Bigr|^2 \in \Ci (  S^2_{\ell_1} \times \ldots \times
S^2_{\ell_n} ) . 
\end{gather}
It is invariant and descends to a function $h_{I, \ell}$ of
$\Ci (\module_\ell)$. Recall that we defined in section \ref{sec:definition_Ha}  a Casimir
operator $Q_{I, \ell}$ which acts on the space $\Hilbert_\ell$ and
that we introduced  in section
\ref{sec:quant_geom} isomorphisms $V_k$ from $\Hilbert_{k \ell -1 }$ to $H^0 (\module_\ell , L_{\ell}^k \otimes
\delta_{\ell} )$. A central result of the paper is the following
theorem.

\begin{theo} \label{theo:main}
 Let $n \geqslant 4$ and $\ell = (\ell_1, \ldots ,
  \ell_n)$ be a family of positive integers satisfying
  (\ref{eq:hyp_existence}), (\ref{eq:hyp2}) and (\ref{eq:hyp1}). For
  any subset $I$ of $\{1,
    \ldots , n\}$, the sequence 
$$ \Bigl( \frac{1}{k^2} V_k Q_{I,k\ell-1 }V_k^{-1}  :  H^0 (\module_\ell , L_{\ell}^k \otimes
\delta_{\ell} )  \rightarrow   H^0 (\module_\ell , L_{\ell}^k \otimes
\delta_{\ell} )  \Bigr)_{k=1,2, \ldots }$$ 
is a Toeplitz operator with principal symbol $h_{I, \ell}$ and
vanishing subprincipal symbol. 
\end{theo}

Let us sketch the proof. First it follows from the results of \cite{BoGu} that
the product of two Toeplitz operators $(T_k)$ and $(S_k)$ is a Toeplitz operators. The
principal and subprincipal symbols $f_0$, $f_1$ of $(T_k S_k)$ may be
computed in terms of the symbols $g_0, g_1$ of $(T_k)$ and $h_0, h_1$ of $(S_k)$ 
$$ f_0 = g_0 h_0 , \qquad f_1 = g_0 h_1 + h_0 g_1 + \frac{1}{2i}  \{
f_0 , g_0 \}, $$
cf. \cite{oim_hf} for a proof of the second formula. 
Second, in the case of a Hamiltonian action $G$ with momentum $\mu$
which satisfies all the assumptions of section
\ref{sec:gq}, the infinitesimal action of $\xi \in \mathfrak{g}$
on $H^0 (M,
L^k \otimes\delta)$ is the Kostant-Souriau operator $ik \preq (
\mu^\xi)$. Since it preserves $H^0 (M, L^k \otimes \delta)$, the
sequence $(\preq (
\mu^\xi))_k$ is a Toeplitz operator. Its principal symbol is $\mu^\xi$
and its subprincipal symbol vanishes. 

Using these two general facts, we deduce from the expression
(\ref{eq:Casimir2}) of the
Casimir operator that the sequence 
$$  \Bigl( \frac{1}{k^2} Q_{I, k\ell-1} : V_{k\ell_1 -1 } \otimes \ldots \otimes V_{ k\ell_n -1}
\rightarrow   V_{k\ell_1 -1 } \otimes \ldots \otimes V_{ k\ell_n -1}
\Bigr)_{ k=1,2, \ldots} $$
is a Toeplitz operator of $S_{\ell_1}^2 \times \ldots \times
S_{\ell_n}^2$. With our normalization for the scalar product of
$\mathfrak{su}(2)$,  its principal symbol is $q_{I,\ell}$. Its
subprincipal symbol vanishes. 

Last step of the proof is the symplectic reduction from $S_{\ell_1}^2 \times \ldots \times
S_{\ell_n}^2 $ to $\module_{\ell}$. We will show in part
\ref{sec:reduct-semi-class} that any invariant Toeplitz operator of a
Hamiltonian space $(M, G)$ descends to a Toeplitz operator of the
quotient $M \varparallel G$. This results was
already proved in \cite{oim_qr} for torus action and the proof given there extends
directly to the case of a compact Lie group. But it does not
give any control on the subprincipal symbol. So we propose
another proof and we will show that the principal
and subprincipal symbols descend to the principal and subprincipal
symbols of the reduced Toeplitz operator. The precise statement is the
corollary \ref{the_corollaire}, which implies theorem
\ref{theo:main}.

\section{On the symplectic geometry of polygon spaces}  \label{sec:sympl-geom-polyg}

Let $n \geqslant 3$ and $(\ell_1, \ldots, \ell_n)$ be a family of positive 
numbers, not necessarily integral. We assume that 
\begin{gather} \label{eq:ineq_stric}
\ell_j < \frac{1}{2} (\ell_1 + \ldots +
\ell_n)
\end{gather}
for any $j =1 , \ldots , n$. Recall that $S^2_{\ell}$ is the
sphere of $\mathfrak{su}(2)^*$ with radius $\ell$. It is the coadjoint
orbit with symplectic volume $2 \pi \ell$.  
Consider the symplectic quotient 
$$ \module_{\ell} = \bigl\{ (x_1, \ldots , x_n ) \in S^2_{\ell_1}
\times \ldots \times S^2 _{\ell_n} / \; x_1 + \ldots + x_n = 0
\bigr\}/ SU(2) $$
Since we do not assume (\ref{eq:hyp2}), $\module_\ell$ may have some
singularities. They form a finite set $S_\ell$ consisting of the classes
$[(x_1, \ldots , x_n)]$ such that all the $x_i$ belong to the same line. $\module_{\ell}
\setminus S_{\ell}$ is a symplectic manifold of dimension $2(n-3)$. It
is non-empty because of the inequalities (\ref{eq:ineq_stric}). 

In section \ref{sec:integrable-system}, we associate to each
admissible graph an integrable system of $\module_\ell$ and give its
properties. The motivation is to obtain next the best description of
the joint eigenstates of the operators $(H_{a, \ell})$. 
In section  \ref{sec:bending-flows-action}, we make clear the
relationship with the decomposition of polygons into triangles. We
also introduce natural action-angle coordinates of the integrable
system. The proofs of the various results are postponed to the next sections.

\subsection{An integrable system}\label{sec:integrable-system}

Let $\Ga$ be an admissible graph with $n$ half-edges. Assume that the internal edges are oriented. Then
for any internal edge $a$, denote by $I(a)$ the set of half-edges
which are connected to the initial vertex of $a$ by a path which
doesn't contain the terminal vertex of $a$. Introduce the length function $\la_a$ of
$\module_{\ell}$ defined by
$$ \la_{a} ([x_1, \ldots , x_n]) = \bigl| \sum _{i \in I(a)} x_{i} \bigr| $$
and denote by $h_a$ the square of $\la_a$. Observe that $\la_a$
is defined independently of the direction of $a$. Using this, it is
easy to see that the Poisson bracket $\{ h_a , h_b \} = 0$ vanishes
for any two internal edges. 

Denote by $\module ^*_{\ell, \Ga}  $ the subset of $\module_\ell
\setminus S_{\ell}$ where none of the functions
$\la_a$ vanishes. The functions $\la_a$ are smooth on
$\module ^*_{\ell, \Ga}$ and they mutually Poisson commute.
To describe the joint image of the $\la_a$, we introduce some more
notations.

Let us denote by $E (\Ga)$  the set of edges of
$\Ga$.  Let $\De(\Ga)$ be the convex polyhedron of $\R^{E (\Ga)}$ which
consists of the $(2n-3)$-tuples $(d_a )$ of non-negative real numbers
such that for any three mutually distinct edges $a$, $b$ and $c$
incident to the same vertex, $d_a$, $d_b$ and $d_c$ are the lengths of
an Euclidean triangle, that is the inequalities 
 $|  d_a - d_b | \leqslant d_c \leqslant d_a + d_b$
hold. Let $\De(\ell, \Ga)$ be the convex polyhedron of $\R^{\Ei
  (\Ga)}$ 
$$ \De(\ell, \Ga) = \{ (d_a)_{a\in \Ei
  (\Ga)} / \;  (\ell_i, d_a)_{i \in \Eh (\Ga), a \in \Ei (\Ga)} \in
  \De (\Ga) \} $$ 
The following theorem will be proved in section  \ref{sec:proof-int-syst}. 

\begin{theo} \label{theo:int_system}
Let $n\geqslant 4$ and let $\Ga$ be an admissible graph
  with $n$ half-edges. Let $(\ell_i)_{i =1, \ldots, n}$ be a family of
  positive real numbers satisfying (\ref{eq:ineq_stric}).  The image of the map 
$$ \la : \module_\ell \rightarrow \R^{\Ei(\Ga)}, \quad x \rightarrow
(\la_a (x) ) $$
is the polyhedron $\Delta(\ell, \Ga)$. The fibers of $\la$ are
connected. 
Denote by $\Int (\De(\ell, \Ga))$ the interior of $\Delta (\ell,
\Ga)$. Then the  set $$\mreg_{\ell, \Ga} :=\la ^{-1} (\Int \De(\ell, \Ga))$$ is
open and dense in
$\module_\ell$ and included in $\module_{\ell,\Ga}^*$. For any $x \in \mreg_{\ell, \Ga}$, the differentials
$d_x \la_a$, ${a\in \Ei (\Ga)}$, are linearly independent.
\end{theo} 
 This result is strongly related to theorem  \ref{sec:spectre} about the
 spectrum of the operators $(H_a)$. The admissible colorings of the
graph $\Ga$ are particular points of the polyhedron $\De ( \Ga)$. The
relationship between a color $\varphi (a)$ and the eigenvalue
$\varphi(a) ( \varphi (a) +1 )$ is similar to the relationship between
$\la_a$ and its square $h_a$. 
These analogies will be made more precise in part
\ref{sec:bohr-somm-cond} where we shall
state the Bohr-Sommerfeld conditions.   
To prepare the description of the joint eigenstates of the $(H_{a, \ell})$ as
Lagrangian sections, we explain now how the Hamiltonian flow of the
$\la_a$ defines a torus action. The following
result is proved in section \ref{sec:proof-tore-action}

\begin{theo} \label{sec:tore_action}
Under the same assumptions as theorem \ref{theo:int_system},  
the Hamiltonian flows on $\module ^*_{\ell, \Ga}$ of the $\frac{1}{2} \la_a$, $a
\in \Ei(\Ga)$,  mutually commute and are $2
\pi$-periodical. Con\-se\-quently they define a torus action 
$$ \tore^{\Ei(\Ga)} \times \module^*_{\ell, \Ga} \rightarrow \module
^*_{\ell, \Ga} $$
This action is free over $\mreg_{\ell, \Ga} $. The fibres of $\la$ in $\mreg_{\la, \Ga}$ are
the torus orbits.
\end{theo}

As we will see in the next section, there exist natural angle
functions which form together with the $\frac{1}{2} \la_a$
an action-angle coordinate system. 
When the $\ell_i$ are integers and satisfy the
condition (\ref{eq:hyp1}), there exist other remarkable action
coordinates $\ga_a$ associated to the prequantum bundle $L_{\ell}$. They are
defined modulo $\Z$ by the condition that $ e ^{2i \pi \ga_a (x) }$ is
the holonomy in $L_{\ell}$ of the loop 
$$ t \in [0, \pi] \rightarrow \Phi_{a,t} (x)$$ 
where $\Phi_{a,t}$ is the Hamiltonian flow of $\la_a$ at time $t$.
We will prove in part \ref{sec:comp-holon} the following 
\begin{prop}  \label{sec:action}
For any internal edge $a$, we have $\ga_a = \frac{1}{2} \bigl( \la_a + \sum_{ i \in I(a)}  \ell_i
\bigr)$. 
\end{prop}

\subsection{Bending flows and action-angle coordinates} \label{sec:bending-flows-action}

Consider first a polygon with $n$ sides numbered from 1
to $n$. Let us cut the polygon into
$n-2$ triangles by connecting the vertices with $n-3$ straight
lines. To this triangulation is associated a admissible graph $\Ga$
defined as follows. The half-edges of $\Ga$ are the sides of the
polygon. The internal edges are the straight lines connecting the
vertices of the polygon. The vertices of $\Ga$ are the triangles. The
edges incident to a vertex are the sides of the triangle, cf. figure
\ref{fig:triangulation}.  It is easily seen
that this graph is admissible. Furthermore any admissible graph may be
obtained in this way.

In the following, we consider a triangulation of a $n$-sided polygon
$P$ with its associated
admissible graph $\Ga$. We assume that the sides of the polygon are
numbered in the apparent order. Then to each $n$-tuple $x= (x_1, \ldots,
x_n)$ of $\mathfrak{su}^*(2)$ we associate the polygon $P(x)$ with vertices $0$,
$x_1 +x_2$, \ldots , $x_1+ \ldots +x_{n-1}$.  So each class $[x_1,
\ldots, x_n]$ of $\module_\ell$ represents a polygons of $\mathfrak{su}
^*(2)$ up to isometry. The triangulation of $P$ induces a
triangulation of $P(x_1, \ldots, x_n)$. Observe that for each internal
edge $a$ of $\Ga$, $\la_{a} (x)$ is the length of a internal edge of
the triangulation. This explains that the image of $\la$ is the
polyhedron $\De ( \ell, \Ga)$ as asserted in theorem
\ref{theo:int_system}. Indeed to each vertex of $\Ga$ 
corresponds a face of the triangulation and related triangle
inequalities.

\begin{figure}
\begin{center}
\begin{picture}(0,0)%
\includegraphics{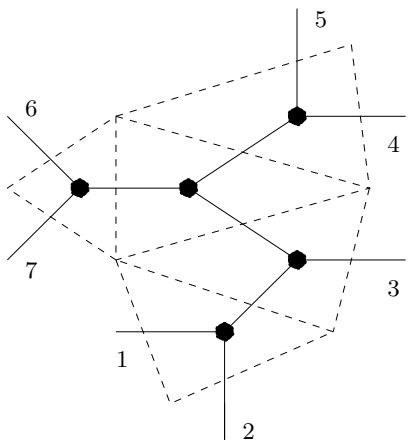}%
\end{picture}%
\setlength{\unitlength}{2960sp}%
\begingroup\makeatletter\ifx\SetFigFont\undefined%
\gdef\SetFigFont#1#2#3#4#5{%
  \reset@font\fontsize{#1}{#2pt}%
  \fontfamily{#3}\fontseries{#4}\fontshape{#5}%
  \selectfont}%
\fi\endgroup%
\begin{picture}(3438,3661)(1789,-3110)
\put(2701,-2461){\makebox(0,0)[lb]{\smash{{\SetFigFont{9}{10.8}{\rmdefault}{\mddefault}{\updefault}{\color[rgb]{0,0,0}$1$}%
}}}}
\put(3751,-3061){\makebox(0,0)[lb]{\smash{{\SetFigFont{9}{10.8}{\rmdefault}{\mddefault}{\updefault}{\color[rgb]{0,0,0}$2$}%
}}}}
\put(4951,-1861){\makebox(0,0)[lb]{\smash{{\SetFigFont{9}{10.8}{\rmdefault}{\mddefault}{\updefault}{\color[rgb]{0,0,0}$3$}%
}}}}
\put(4951,-661){\makebox(0,0)[lb]{\smash{{\SetFigFont{9}{10.8}{\rmdefault}{\mddefault}{\updefault}{\color[rgb]{0,0,0}$4$}%
}}}}
\put(4351,389){\makebox(0,0)[lb]{\smash{{\SetFigFont{9}{10.8}{\rmdefault}{\mddefault}{\updefault}{\color[rgb]{0,0,0}$5$}%
}}}}
\put(1951,-361){\makebox(0,0)[lb]{\smash{{\SetFigFont{9}{10.8}{\rmdefault}{\mddefault}{\updefault}{\color[rgb]{0,0,0}$6$}%
}}}}
\put(1951,-1711){\makebox(0,0)[lb]{\smash{{\SetFigFont{9}{10.8}{\rmdefault}{\mddefault}{\updefault}{\color[rgb]{0,0,0}$7$}%
}}}}
\end{picture}%
\caption{Triangulation} \label{fig:triangulation}
\end{center}
\end{figure}

In the proof of theorem \ref{sec:tore_action}, we shall show that the
Hamiltonian flow of $\la_a$ at time $t$ maps $x = [x_1, \ldots , x_n] \in
\module^*_\ell$ into  $\Phi_{a} (x,t) = [z_1 , \ldots, z_n]$ where 
$$  z
_i = \begin{cases} \Ad_g x_i \text{ if $i \in I(a)$} \\ x_i \text{
    otherwise}
\end{cases}, \qquad  \text{with } g = \exp \Bigl( \frac{t}{\la_a (x)}  \sum_{i \in
  I(a)} x_i \Bigr). $$    
Here we identified $\mathfrak{su} (2)$ and $\mathfrak{su} (2) ^* $ by using the
invariant scalar product introduced above. 
Considering the previous
interpretation of $\module_\ell $, these flows are
bendings along the internal edges of the triangulation. 

Now for any internal edge $a$, introduce a coordinate 
$$\te_a : \mreg_{\ell, \Ga} \rightarrow \R/ 2 \pi \Z$$ which measures
the dihedral angle between the two faces adjacent to the internal edge
of the triangulation associated to $a$. More precisely, we require
that $\te_a = 0$ or $\pi$ when these faces are coplanar and that
$\te_a(\Phi_{a} (x,t)) = \te_a + t$. 

\begin{theo} \label{sec:lagrangian} 
The subspace $\planar_\ell$ of $\module _\ell \setminus S_{\ell}$ which consists of
classes of $n$-tuples of coplanar vectors, is a closed Lagrangian
submanifold.  Consequently, the family $\frac{1}{2} \la_a, \te_a$, $a \in \Ei (\Ga)$, is an
action-angle coordinate system, that is the symplectic form is given
by $$\om = \sum_{ a \in \Ei ( \Ga)} d (\tfrac{1}{2} \la_a)
\wedge d \te_a$$ 
on $\mreg_{\ell, \Ga}$.
\end{theo}
This theorem was proved in \cite{KaMi} for the graph $\Ga_n$ and its
associated triangulation given in figures \ref{fig:gamma_n} and \ref{fig:pol_gamma_n}.
\begin{proof} 
It is easily seen that $\planar_\ell$ is a $(n-3)$-dimensional
submanifold of $\module _\ell \setminus S_{\ell}$. Let $R$ be a
reflexion of $\mathfrak{su}^* (2)$. 
Consider the involution $\Psi$ of $\module _\ell \setminus S_{\ell}$
which sends $[x_1, \ldots, x_n]$ into $[R(x_1), \ldots , R(x_n)]$. It
transforms the symplectic form into its opposite. $\planar_{\ell}$
being the fixed point set of $\Psi$, it is a Lagrangian submanifold. 
  \end{proof}

\subsection{Proof of theorem \ref{theo:int_system}} \label{sec:proof-int-syst}

Let $H $ be the subspace of $(\mathfrak{su}(2))^n$
\begin{gather} \label{eq:defH}
 H := \{ (x_i ) \in (\mathfrak{su}(2))^n /  \; x_1 + \ldots + x_n = 0
\}
\end{gather}
Denote by $E(\Ga)$ the set of edges of $\Ga$ and introduce the map $\underline{\la} : H \rightarrow \R^{E(\Ga)}$ whose components are  
\begin{gather} \label{eq:defla}
 \underline{\la} _a(x) = \begin{cases} 
\bigl| \textstyle{\sum}
_{i \in I(a)} x_i \bigr|  \qquad \text{if $ a$ is an internal edge,} \\
|x_a| \qquad \text{ if $a$
is an half-edge.}
\end{cases}
\end{gather}
Recall that $\De (\Ga)$ is the convex polyhedron of $\R^{ E(\Ga)}$
consisting of the $(d_a)$ such that for any three edges $a,b$, and $c$
mutually distinct and incident to the same vertex, $d_a$, $d_b$ and $d_c$ satisfy the
triangle inequalities.

\begin{theo} \label{theo:image}
The image of $\underline{\la}$ is the set $\De(\Ga)$. The fibres
  of $\underline {\la}$ are connected.  Furthermore the interior points of
  $\De (\Ga)$ are regular values of $\underline{\la}$. 
\end{theo}

The interior of $\De ( \Ga)$ consists of the $(d_a)_{a
  \in E( \Ga)} $ such
that the strict triangle inequalities 
$$| d_a - d _b | < d_c < d_a + d _b$$ hold at each
vertex. So the functions $\underline{\la}_a$ do not vanish and are
smooth on
$\underline{\la}^{-1}( \operatorname{Int} (\De (\Ga)))$.

Before the proof, let us deduce theorem \ref{theo:int_system} from \ref{theo:image}. We
immediately have that the image of $\la$ is the polyhedron $\De
( \ell, \Ga)$ and that its fibre are connected. Next it is easily
checked  that
the interior of $\De ( \ell, \Ga)$ is the set of $(d_a )_{ a \in \Ei
  (\Ga)}$ such that $(d_a, \ell_i)$ belongs to $\Int ( \De)$. So
theorem \ref{theo:image} implies that the differentials of the $\la_a$
are linearly independent over $ \la ^{-1} (\Int \De ( \ell, \Ga))$. 
It remains to show the density of $\la ^{-1} (\Int \De ( \ell, \Ga))$,
which follows from proposition \ref{sec:densite}.

\begin{proof}[Proof of theorem \ref{theo:image}] The proof is by
  induction on the number of half-edges. The result is easily proved for the graph 
  with three monovalent vertices. Observe that in this
 case the fibers of $\underline{\la}$  are the orbits of the
 diagonal $SU(2)$-action.  Furthermore $(x_1, x_2, x_3)$ is a regular
 point of $\underline{h}$ if the triangle with vertices $0, x_1, x_1 +
 x_2$ is non-degenerate. 

Let $n \geqslant
  3$. Consider an admissible graph $\Ga$ with
  $(n+1)$ half-edges. 
Then $\Ga$
admits a vertex $\underline{v}$ incident to two half-edges and an internal edge. To see this,
consider the graph obtained from $\Ga$ by deleting all the
half-edges. This new graph being a tree, it has at least two
monovalent vertices, these vertices satisfy the property. Next
we order the half-edges of $\Ga$ in
such a way that the two ones incident to $\underline{v}$ are the $n$-th and
$(n+1)$-th.  Denote by $\underline{a}$ the internal edge incident to $\underline{v}$.
Changing the directions of the internal edges if necessary, we may assume that 
$$I(\underline{a} ) = \{ n, n+1 \} , \qquad I (a) \subset \{ 1, \ldots
, n-1 \} .$$ 
Let $\Ga'$ be the graph obtained from $\Ga$ by removing the vertex
$\underline{v}$ and 
  the $n$-th and $(n+1)$-th half-edges. Denote by $\underline{\la}'$ and $\underline{\la}''$
  the maps associated respectively to $\Ga'$
  and the graph $\Ga''$   
  with three monovalent vertices.
Then the induction is based on the simple observation that
$\underline{\la}
 (x ) = d$ if and only if 
$$ \underline \la ' (x_1, \ldots, x_{n-1}, -y ) = d |_{E(\Ga')}  \quad \text{
  and } \quad \underline \la '' (y, x_n , x_{n+1} ) = ( |y|, d_n, d_{n+1}) $$
where $y =  x_1 + \ldots + x_{n-1}$. As a first consequence, we immediately deduce
that the image of $\la$ is $\De(\Ga)$ if we already know that the images of $\la'$ and
$\la''$ are the polyhedra $\De(\Ga')$ and $\De ( \Ga'')$. 

Let us prove that the fibers of
$\underline \la$ are connected. Let $x^0$ and $x^1$ be such that $\underline {\la}
(x^0) = \underline {\la}
(x^1)$. Denote by $x'$ the $(n-1)$-tuple $(x_1, \ldots x_{n-1})$ and by
$y$ the sum $ x_1 + \ldots + x_{n-1}$.  Assuming the fibers of $\underline{\la}'$ are connected, there
exists a path $$t \rightarrow ( x'^t,- y^t)$$ from $(x'^0, -y^0)$ to $(x'^1, -y^1)$ remaining in
the same fiber of $\underline \la'$. Since $|y^t|$ is constant, there
exists a path $ t \rightarrow g^t$ in $SU (2)$ such that $g^t y^0 = y^t$. Then the
path $$t \rightarrow ( x'^t, g^t x^0_n, g^t x^0_{n+1})  $$
connects $ x^0$ with $(x'^1, g^1 x^0_n, g^1 x^0_{n+1})$ and takes its
value in a single fiber of $\underline{\la}$. Next since 
$$\underline \la'' (-g^1(
x^0_n +  x^0_{n+1}),g^1
x^0_n, g^1 x^0_{n+1} ) = \underline \la '' (-(x^1_n + x^1_{n+1}), x^1_n, x^1_{n+1}),$$
 there exists a path $$ t \rightarrow (-(z_n^t +
z_{n+1}^t), z_n ^t, z_{n+1}^t)$$ which connects these two points by
remaining in the same fiber of $\underline \la ''$. Furthermore since $ g^{1}
(x^0_n + x^0 _{n+1}) = y^1 =x^1_n + x^1_{n+1}$, one may choose this path in such a way that $(z_n^t +
z_{n+1}^t)$ is constant. Finally the path 
$$ t \rightarrow ( x'^1, z^t_n, z^t_{n+1})  $$
connects $ (x'^1, g^1 x^0_n, g^1 x^0_{n+1})$ with $x^1$ and this ends
the proof of the connectedness.  

Let us prove that the interior points of $\De( \Ga) $ are regular values of
$\underline \la$. Assume that the result is satisfied for the function
$\underline{\la}'$ associated to $\Ga'$. Let $x^0$ be such that its
image belongs to the interior of $\De (\Ga)$. Then one has to prove
that the map 
$$  (\mathfrak{su}(2))^{n+1} \rightarrow \R^{E(\Ga) } \times \mathfrak{su}(2)
, \qquad x \rightarrow ( \underline{\la} (x), x_1 + \ldots + x_{n+1}) $$
is submersive at $x^0$, 
where $\underline \la $ is the obvious extension of $\underline \la $ from $H$ to
$(\mathfrak{su}(2))^{n+1}$. 
Consider the isomorphism
$$ \Psi: (\mathfrak{su}(2))^{n+1} \rightarrow H' \times H'' , \qquad x
\rightarrow (x' , -y) , ( - (x_n + x_{n+1} ) , x_n, x_{n+1})$$
Here we denote as previously  by $x'$ the $(n-1)$-tuple $(x_1, \ldots x_{n-1})$ and by
$y$ the sum $ x_1 + \ldots + x_{n-1}$. 
Then one has 
$$\underline \la'_a \circ \Psi  =
\underline \la _a, \quad \forall a \in E(\Ga') \quad \text{ and } \quad
\underline \la''
\circ \Psi  = ( \underline \la _{\underline{a}} , \underline \la_n,
\underline \la_{n+1})  $$
By induction assumption, $\underline \la '$ is submersive at $(x'^0,
-y^0)$. Hence it suffices to prove that the map  
$$   H ''
\rightarrow \R^2 \times \mathfrak{su}(2), \qquad  (-(x_n + x_{n+1}),
x_n, x_{n+1} ) \rightarrow  (| x_n|, | x_{n+1}| , x_n + x_{n+1}) $$ 
is submersive at $ (-(x_n^0 + x_{n+1}^0),
x_n^0, x_{n+1}^0 )$. This is true because $ \underline \la (x^0)$ being an interior
point of $\De (\Ga)$, the triangle with vertices $0$, $x_n^0$ and
$x_{n}^0 + x_{n+1}^0$ is non-degenerate. This ends the proof of the theorem.
\end{proof}

For any $n$-tuple $\ell$ of positive numbers, denote by $P_{\ell}$ the
subset of $H$,
$$ P_\ell: =   \bigl(S_{\ell_1}^2
\times \ldots \times S_{\ell_{n}}^2 \bigr) \cap H, $$ 
so that $\module_{\ell}$ is the quotient of $P_{\ell}$ by $SU(2)$. 
\begin{prop} \label{sec:densite}
If $\ell$ satisfies 
  (\ref{eq:ineq_stric}), then $\underline{\la} ^{-1} ( \De ( \Ga )) \cap
  P_\ell$ is dense in $P_{\ell}$. 
\end{prop}

\begin{proof} 
Again the proof is by induction on the number of half-edges. Introduce the
graphs $\Ga$, $\Ga'$ and $\Ga''$ as in the proof of theorem
\ref{theo:image} and assume the result is satisfied for $\Ga'$ and for
any $\ell$ satisfying the inequalities (\ref{eq:ineq_stric}). 

Let $\ell$ be a $(n+1)$-tuple satisfying (\ref{eq:ineq_stric}). Let $\tilde {S}_{\ell}$ be the set of $(x_1,
\ldots , x_{n+1}) \in P_{\ell}$ such that the $x_i$ are mutually
colinear. It is easily deduced from (\ref{eq:ineq_stric}) that $P_{\ell} \setminus
\tilde{S}_\ell$ is dense in $P_{\ell}$.  

Next we consider the set
$Q$ consisting of the $(x_1, \ldots , x_{n+1}) \in P_{\ell} \setminus
\tilde{S}_\ell$ such that the lengths $\ell'_n = | x_n + x_{n+1}|$
satisfy the two following conditions. First $(\ell_1, \ldots , \ell_{n-1},
\ell'_n)$ satisfies the inequalities (\ref{eq:ineq_stric}) and
second $\ell'_n, \ell_n, \ell_{n+1}$ satisfy the strict triangle
inequalities. One proves that $Q$ is dense in $P_\ell \setminus
\tilde{S}_\ell$. To do this observe that the inequalities
(\ref{eq:ineq_stric}) are satisfied as soon as there is no equality
and this can happen only for a finite number of $\ell'_n$. For any $x^0 \in
P_{\ell} \setminus \tilde{S}_\ell$, one has to find points in $Q$
arbitrarily close to $x^0$. This can be proved by considering
separetly the case where $x^0_1, \ldots, x^0_{n-1}$
are mutually colinear. If there are not, one concludes by using
that the map 
$$  S_{\ell_1} \times \ldots \times S_{\ell_{n-1}} \rightarrow
\mathfrak{su}(2)^*$$ 
if submersive at $(x^0_1, \ldots , x^0_{n-1})$. 

The last step is to prove the density of $\underline{\la} ^{-1} ( \De ( \Ga )) \cap
  P_\ell$ in $Q$. This follows from the induction assumption. 
\end{proof}

\subsection{Proof of theorem \ref{sec:tore_action}} \label{sec:proof-tore-action}

 The Hamiltonian flow of the $\la_a$  is easily described
applying the following general result. 
Let $(M,\om)$ be a symplectic manifold with a
Hamiltonian action of a Lie group $G$.  Assume that the Lie algebra $\mathfrak{g}$ has an
invariant scalar product, which we use to identify $\mathfrak{g}$ and
$\mathfrak{g}^*$. Denote by $\mu$ the momentum of the action and 
let $M^*$ be the open set $\{ \mu \neq 0 \}$ of $M$. One easily checks
the following proposition. 

\begin{prop} \label{prop:flot_norm}
The Hamiltonian flow of $|\mu| \in \Ci ( M^*)$ at time $t$ is given by 
$$ \Phi_t (x) = \exp \Bigl( t \frac{\mu (x) }{|\mu (x) |} \Bigr).x, \qquad
\forall x \in M^* $$
\end{prop}

Let $I$ be a subset of $\{ 1, \ldots, n \}$. Consider the Hamiltonian action of $SU(2)$ on $S^2_{\ell_1} \times
\ldots \times S^2_{\ell_n}$ with momentum 
\begin{gather} \label{eq:momentI} 
\mu_{I}( x_1, \ldots
, x_n ) = \sum _{i \in I} x_i
\end{gather}
By the previous proposition, its flows $\rho_{I,t}$ at time $t$ sends
$(x_1, \ldots , x_n)$ into the $n$-tuple $(y_1, \ldots , y_n)$ given by
$$y_i = \begin{cases}  \Ad_g x_{i} \text{ if } i \in I, \\ 
x_i \text{ otherwise}, 
\end{cases} 
\text{ with  }  g = \exp \Bigl( t \frac{ \mu_{I}(x) }{| \mu_{I}(x) |}
\Bigr). $$    
We apply this to the set $I(a)$ associated to an internal edge $a$. The function $|\mu_{I(a)}|$ is invariant by the diagonal action and
descends to the function $\la_a$. Hence the flows $\rho_{I(a),t}$
lifts the Hamiltonian flow of
$\la_a$. Furthermore, with our
normalization, 
 $$\exp( \xi) =1 \Leftrightarrow |\xi| \in 2 \pi \N, \qquad \forall
 \xi \in \mathfrak{su}(2). $$
Since the coadjoint action factorizes through a $SO(3)$-action, the flow of the $\la_a$ is $\pi$-periodical and this
proves the first part of theorem \ref{sec:tore_action}. 
Next by theorem \ref{theo:int_system}, the torus orbits in
$\mreg_{\ell, \Ga }$ and
the fibers of $\la$ have
the same dimension. Furthermore the fibers are
connected. Hence the fibers in $\mreg_{\ell, \Ga}$ are the torus orbits. It remains
to prove that the torus action is free in $\mreg_{\ell, \Ga}$. 

Let $H$ be the subspace defined in (\ref{eq:defH}). Let us extend
the action $\rho_{I,t}$ on $H \setminus \{ \sum_{i \in I } x_i \neq 0
\}$ in the obvious way. We will consider the actions $\rho_{I(a), t}$
  altogether. 
Here we can not choose the
orientations of the internal edges in such a way that the sets $I(a)$
are mutually disjoint. So the action of $ \rho_{I(a),t}$ and $ \rho_{I(b),t}$ do not
necessarily commute.  Let $a_1$, \ldots , $a_{n-3}$ be
the internal edges of $\Ga$. The following proposition completes the proof of
theorem \ref{sec:tore_action}.

\begin{prop} For any $x \in H$ such that $\underline \la (x)  \in \Int
  \De (\Ga)$ and for any $(t_i) \in
  \R^{ n-3}$, if there exists $g \in SU(2)$ such that
$$ \rho_{I(a_1),t_1} (\rho_{I(a_2),t_2}( \ldots (\rho_{I(a_{n-3}),t_{n-3}}(x)) \ldots) ) = g.x $$
then $t_1 \equiv \ldots \equiv t_n \equiv 0$ modulo $\pi$. 
\end{prop}

\begin{proof}
Again the proof is by induction on the number of half-edges. Introduce the
graphs $\Ga$, $\Ga'$ and $\Ga''$ as in the proof of theorem
\ref{theo:image} and assume the result is satisfied for $\Ga'$. Since
the various circle actions mutually commute modulo the diagonal action
of $SU(2)$, the result is independent of the order of the internal
edges. Observe also that
$$ \rho_{I,t}(x) = g(x,t) . \rho_{I^c, t} (x)
$$
So the result does not depend on the direction of the internal edge. 
So we may assume that $a_{n-2}$  is the internal edge of $\Ga$ incident to the $n$-th and
$(n+1)$-th half-edges. We may also assume that $I(a_{n-2})= \{ n, n+1 \}$
and $I(a_k) \subset \{1 , \ldots, n-1 \}$ for $k=1, \ldots, n-3$.   
Let $x
\in  \underline \la ^{-1} (\Int
  \De (\Ga) )$ such that 
$$  \rho_{I(a_1),t_1} (\rho_{I(a_2),t_2}( \ldots
(\rho_{I(a_{n-2}),t_{n-2}}(x) ) \ldots) ) = g.x .$$
This implies that
\begin{gather} \label{eq:9}   
\rho_{I(a_1),t_1} (\rho_{I(a_2),t_2}( \ldots
(\rho_{I(a_{n-3}),t_{n-3}}(x',-y) ) \ldots) ) = g.(x', -y)
\end{gather} 
where $x' = (x_1,\ldots, x_{n-1})$ and $y = x_1 + \ldots +
x_{n-1}$. Furthermore
\begin{gather}  \label{eq:10}
 \exp \Bigr( t_{n-2} \frac{ x_{n} + x_{n+1} }{|x_n + x_{n+1}|} \Bigl)
.(x_n,x_{n+1}) = g.(x_n, x_{n+1})
\end{gather}
By the induction assumption, one has that $t_1 \equiv \ldots  \equiv t_{n-3}
\equiv 0$ modulo $\pi$. Consequently, (\ref{eq:9}) reads as
$$  (x_1,\ldots, x_{n-1}) = g. (x_1,\ldots, x_{n-1})$$
If $g \neq \pm \id$, this implies that the vectors $x_1, \ldots, x_{n-1}$ are
colinear and this contradicts the fact that $\underline{\la}(x) \in \Int ( \De (
\Ga))$. So $g = \pm \id$. Finally since $x_n$ and $x_{n+1}$ are not
colinear, equation (\ref{eq:10}) implies that $t_{n-2} \equiv 0$
modulo $\pi$. 
\end{proof}

\subsection{Computation of the holonomies} \label{sec:comp-holon}

Extending the proposition \ref{prop:flot_norm} to the prequantum case, we can also
compute the actions $\ga_a$. Assume that $(M, \om)$  admits a prequantization
bundle $L \rightarrow M$ with curvature $\frac{1}{i} \om$  and that the
action of $G$ lifts to $L$ satisfying the usual assumption
(\ref{eq:lift_action}). 
\begin{prop} 
 For any $x \in M^*$ and $u \in L_x$, one has
$$ e^{-it | \mu (x) | } \mathcal{T}_t.u = \exp \Bigl( t \frac{\mu (x)
}{|\mu (x) |} \Bigr).u $$
where $\mathcal{T}_t.u \in L_{\Phi_t (x)}$ is the parallel transport of
$u$ along the path $ s\in [0,t] \rightarrow \Phi_s (x) $. 
\end{prop}
To deduce proposition \ref{sec:action}, we apply this result at time
$\pi$ to the Hamiltonian action with momentum $\mu_{I(a), t}$
  introduced in (\ref{eq:momentI}).  Be careful that the action on the
prequantization bundle factorizes through a $SO (3)$-action only if $
\sum_{i \in I(a)} \ell_i$ is even.

\section{Semi-classical properties of the joint eigenstates} \label{sec:semi-class-prop}
 
Let $\Ga$ be an admissible graph with $n$ half-edges and let $(\ell_i)_{i
  =1 , \ldots , n}$ be a family of positive 
integers. For any integer $k \geqslant 1$, we defined in section
\ref{sec:definition_Ha} a family $$ \{ H_{a, k \ell-1}; \;  a \in
\Ei ( \Ga)\}  $$ of mutually commuting operators of
$\Hilbert_{k \ell -1}$. For any joint eigenvalue $E=(E_a) \in
\R^{\Ei(\Ga)}$ of the operators $k^{-2}H_{a,k\ell-1}$, introduce a unitary
eigenvector $\Psi_{E,k}$
$$ \frac{1}{k^{2}}H_{a, k \ell -1} \Psi_{E,k} = E_a \Psi_{E,k}, \qquad \forall a \in \Ei(\Ga) .$$
 Assume furthermore that $\ell$ satisfies assumptions
 (\ref{eq:hyp_existence}), (\ref{eq:hyp2})
and (\ref{eq:hyp1}). Then by theorem \ref{theo:geom-real-space}, $\Hilbert_{k \ell -1}$ is isomorphic to $H^0
( \module_\ell, L_\ell^k \otimes \delta_\ell)$ and by theorem
\ref{theo:main} the sequence 
$$ T_a = \Bigl( \frac{1}{k^2} V_k H_{a, k \ell -1} V_k^{-1} \Bigr)_{k
  =1,2, \ldots} $$
is a Toeplitz operator with principal symbol 
$$ h_a  = \bigl| \sum _{i \in I(a)} x_{i} \bigr|^2. $$
As a consequence we can describe precisely the eigenstates
$\Psi_{E,k}$ in the semi-classical limit. This is done in section \ref{sec:etats-propres}. The
next section is devoted to the Bohr-Sommerfeld conditions. 

\subsection{Eigenstates of the $H_a$} \label{sec:etats-propres}
First the eigenstates are microlocalised on
the level sets of the joint principal symbol $h:=(h_a) : \module_\ell
\rightarrow \R^{\Ei(\Ga)}$. Introduce the closed subset 
$$ \Lambda := \{ (x, E)/ \; h(x) = E    \}  $$ 
of $\module_\ell \times \R^{\Ei ( \Ga) }$.
\begin{theo} \label{theorem:microsupport}
For any  $ (x, E) \notin \Lambda$, there exist neighborhoods $U
$ of $x$ in $\module_\ell$ and $V$ of $E$ in $\R^{\Ei ( \Ga) } $ and a sequence $(C_N)$
of positive real numbers such that    
$$ | V_{k} \Psi_{E,k}  (y ) |  \leqslant C_N k^{-N} $$ 
for any integers $k$ and $N$, any  $y \in U$ and any joint eigenvalue
$E \in V$. The same result holds if we replace $V_{k}$ by $V_{k} (
V_{k} ^* V_{k})^{-\frac{1}{2}}$.
\end{theo}
Next since the family $(h_a)$ is an integrable system as stated in
theorem \ref{theo:int_system}, we
can determine modulo $O(k^{-\infty})$ the eigenvectors $  \Psi_{E,k}$ on a neighborhood of
the the level set $h^{-1}(E)$ when $E$ is regular value of $h$.

Consider as in theorem \ref{theo:int_system} the regular set
$\mreg_{\ell, \Ga}
=\la ^{-1} (\Int \De(\ell, \Ga))$. To describe uniformly the various
eigenstates, introduce the submanifold 
$$ \Lambda^{\operatorname {reg}}  := \La \cap (\mreg_{\ell, \Ga} \times \R^{\Ei (
  \Ga)} )$$
of $\module_\ell \times \R^ {\Ei ( \Ga)}$. Denote by $L$ and
$\delta$ the  pull-backs of the bundles $L_{\ell}$
and $\delta_{\ell}$
by the projection $ \module_\ell \times \R^{\Ei (\Ga) } \rightarrow \module_\ell $.

Let $x \in \mreg$. Then $E = h(x)$ is a regular value of $h$ and $h^{-1}
(E)$ is a Lagrangian torus described in theorem
\ref{sec:tore_action}. Let us introduce the isomorphism
\begin{gather} \label{eq:2}
  \varphi_x :    \delta^{2}_x \rightarrow \wedge^{{\operatorname{top} },0} (T^*_x \module)
\rightarrow \wedge ^{\operatorname{top} } (T_x h^{-1} (E) )  \otimes \C 
\end{gather}
Here the first arrow is the isomorphism making $\delta$ a square root
of the canonical bundle and the second arrow is the restriction from $T_x \module $ to $T_x
h^{-1} (E)$. Finally denote by $\mu_E $ the volume element of 
$h^{-1} (E)$ which is invariant by the torus action generated by the
$\la_a$ and normalised by 
\begin{gather} \label{eq:normalisation}
\int_{h^{-1}
(E)} \mu_E = 1.
\end{gather}
Since there is no prefered orientation of $h^{-1}(E)$, $\mu_E$ is
uniquely defined up to a sign.

\begin{theo} \label{theo:quasi_mode}
For any  $(x, E) \in \Lambda^{\operatorname {reg}}$, there exist
neighborhoods $U$ of $x$ in $\mreg_{\ell, \Ga}$ and  $V$ of $E$ in
$\R^{\Ei(\Ga)}$, a section $F$
  of $L \rightarrow U \times V$, a sequence $(g( \cdot, k))_k$ of
  sections of $\delta
  \rightarrow U \times V$, a sequence $(D_k)$ of complex numbers and a
  sequence 
  $(C_N)$ of positive numbers such that for any integer $k$, any
  integer $N$, any $y \in U$ and any joint eigenvalue $E \in V$,
\begin{gather} \label{eq:approx}
 \Bigl| V_k \Psi_{E,k}  (y )  - D_k \Bigl(\frac{k}{2\pi}
\Bigr)^\frac{\dim \module}{8} F^k (y,E) g(
y, E,k)  \Bigr|  \leqslant C_N k^{-N}
\end{gather}
Furthermore
\begin{enumerate}
\item $| D_k | = 1 + O( k^{-\infty})$
\item For any $E$, the restriction of $F$ to $h^{-1}(E) \times \{E \}$
  is flat with a constant unitary norm,  $\overline{\partial} F $
  vanishes to any order along 
  $\La^{\operatorname {reg}}$ and for all  $x \notin \La$, $|F(x)| < 1$.

\item the sequence $(g( \cdot, k))_k$ admits an asymptotic expansion  
$$ g(y,E,  k) = g_0 (y,E) + k^{-1} g_1 (y,E ) + k^{-2} g_2 ( y,E) + \ldots
$$
for the $\Ci$ topology, where each coefficient  $g_i \in \Ci ( U \times
V , \delta)$ is such that 
$\overline{\partial}  g_i$ vanishes to any order along $ \La^{\operatorname {reg}}$. 

\item the restriction of the first coefficient $g_0$ to
  $\La^{\operatorname {reg}}$ is such that 
$$ \varphi_y \bigl( (g_0 (x,E) )^{\otimes 2}\bigr) = \mu_E (y) ,
\qquad \forall (y, E) \in \La^{\operatorname {reg}}  $$
where  $\varphi_y$ and $\mu_E$ are defined in (\ref{eq:2}) and (\ref{eq:normalisation}).
\end{enumerate}
The same result holds   if we replace $V_{k}$ by $V_{k} (
V_{k} ^* V_{k})^{-\frac{1}{2}}$.
\end{theo}

This follows from section 3 of \cite{oim_qm} and section 3 of
\cite{oim_hf}. In the following remarks, we discuss the various
components of the result and how it should be modified in the case
where the subprincipal symbols do not vanish.  

\begin{rem} \label{rem:app}
The estimations  (\ref{eq:approx}) are controlled by the
  restrictions  of $F$ and the 
$g_i$ to  $\La^{\operatorname {reg}}$. More precisely consider
a section $\tilde{F}$ of  $L \rightarrow U \times V$ and a sequence 
$(\tilde{g} (\cdot,k ))$ of sections of $\delta \rightarrow U \times
V$ which satisfy conditions 2. and 3. of theorem \ref{theo:quasi_mode}, then the equalities 
$$ F |_{\La^{\operatorname {reg}}} =
  \tilde{F}|_{\La^{\operatorname {reg}}} \quad \text{ and } \quad g_i|_{\La^{\operatorname {reg}}} = \tilde
g_i|_{\La^{\operatorname {reg}}} \quad \text{ for } \quad i = 1 ,
\ldots , N$$
imply that 
$$ \tilde{F}^k (x, E) \tilde{g} (x, E , k) =   F^k (x, E)
g (x, E , k)  + O(k^{-N-1} )$$ 
uniformly on any compact set of $U \times V$. \qed
\end{rem}

\begin{rem} \label{rem:symbol}
The restriction of $F$ to $\La^{\operatorname {reg}}$ is determined up
to a factor $f(E)$ by the condition that it is flat along $h^{-1}(E)
\times \{E \}$ for any $E$. Equivalently, this condition may be written  as
$$ \nabla^L_{X_a} F = 0 \quad \text{ on } \La^{\operatorname {reg}},
\qquad \forall a \in \Ei ( \Ga)$$
where $X_a$ is the Hamiltonian vector field of $h_a$. In a similar way
the function $g_0$ satisfies the transport equations
\begin{gather} \label{eq:transport}
 \dLie^\delta_{X_a} g_0 = 0  \quad \text{ on }  \La^{\operatorname {reg}},
\qquad \forall a \in \Ei ( \Ga) .
\end{gather}
Furthermore the sections $F$ and $g_0$ are normalized in condition
2. and 4. in such a way that $|D_k|= 1 + O(k^{-1})$.  This follows from theorem 3.2 of \cite{oim_hf} and condition
(\ref{eq:normalisation}).  \qed
\end{rem}

\begin{rem} 
If the subprincipal symbols $h_a^1$ of the $T_a$ had not vanish, we
should replace the transport equations (\ref{eq:transport}) by 
$$\dLie^\delta_{X_a} g_0 + i  h^1_a g_0 = 0  \quad   \text{ on } \La^{\operatorname {reg}},
\qquad \forall a \in \Ei ( \Ga) . \qed$$ 
\end{rem}

\subsection{Bohr-Sommerfeld conditions} \label{sec:bohr-somm-cond}

Bohr-Sommerfeld conditions are the conditions to patch together the
sections  $$F^k( \cdot, E) g_0 ( \cdot,
E)$$ of theorem \ref{theo:quasi_mode} modulo an error
$O(k^{\frac{\operatorname{dim} \module }{8} -1})$. By remarks
\ref{rem:app} and \ref{rem:symbol}, this can be achieved if and only if the restriction to $h^{-1}
(E)$ of the bundle $L_\ell ^k \otimes
\delta_\ell $ is trivial as a flat bundle. Here we consider the flat
structure of $$\delta_\ell
\rightarrow h^{-1} (E)$$ defined in such a way that the local
sections $g_0 (\cdot,
E)$ satisfying  condition 4. of theorem \ref{theo:quasi_mode} are
flat. The flat structure of 
$$L^k_\ell \rightarrow h^{-1}(E)$$ 
is the one induced by the connection of the prequantum bundle. 
It is flat because the curvature of $L$ is the symplectic
form and $h^{-1}(E)$ is Lagrangian. 

For any internal edge $a$ of $\Ga$, consider an integral curve of the
hamiltonian flow of $h_a$ in $h^{-1}(E)$. The family of these curves is a base of the first
group of homology of $h^{-1} (E)$. One deduces the holonomy of
these curves from proposition \ref{sec:action}. We obtain the
Bohr-Sommerfeld conditions.  

\begin{prop} 
 The restriction to $h^{-1}
(E)$ of the bundle $L_\ell ^k \otimes
\delta_\ell $ is a trivial flat bundle if and only if 
\begin{gather} \label{eq:BS_conditions}
  \frac{k}{2} \Bigl( E_a^{\frac{1}{2}} +  \sum_{ i \in I(a)}  \ell_i \Bigr) +
  \frac{1}{2} \epsilon_a (E) \in \Z , \qquad \forall a \in \Ei(\Ga)
\end{gather}
where $\epsilon_a (E)$
is equal to  $0$ or $1$ according to whether the restriction of $\delta_\ell$
to the integral curve of $h_a$ is trivial or not. 
\end{prop}

The indices $\epsilon_a(E)$ 
replace the Maslov indices for the usual Bohr-Sommerfeld conditions. 
It is proved in section 3 of \cite{oim_qm} (cf. also section 3 of
\cite{oim_hf}) that for any couple $(E,k)$ which satisfies the Bohr-Sommerfeld condition, there exists a joint eigenvalue $\tilde{E}
(E,k)$ of the operators $ T_a$ such that  
$$ \tilde{E} (E,k) = E + O ( k^{-2})$$
Let us compare this with the description of the spectrum of the $H_{a
  , k \ell -1}$ given in theorem \ref{sec:spectre}. The joint
eigenvalues are indexed by  the admissible colorings $\varphi$ of $\Ga$
such that $\varphi  (i) = k \ell_i -1$ for any half-edge $i$. Observe
that the parity conditions satisfied at each vertex are alltogether equivalent to  
\begin{gather} \label{eq:BS_coloriage_v2_1}
 \varphi ( 1 ) +   \varphi ( 2 ) + \ldots +
 \varphi ( n ) \in 2 \Z
\end{gather}
and 
\begin{gather} \label{eq:BS_coloriage_v2}
 \varphi (a)  + \sum_{i \in I(a)}    \varphi ( i ) \in 2 \Z
\end{gather}
for any internal edge $a$ of $\Ga$. Here the equation
(\ref{eq:BS_coloriage_v2_1}) is satisfied because $\varphi  (i) = k
\ell_i -1$ and we assumed (\ref{eq:hyp1}). Comparing the
Bohr-Sommerfeld conditions (\ref{eq:BS_conditions})
with the equations (\ref{eq:BS_coloriage_v2}), we obtain the following
\begin{prop}  
$ \epsilon_a = \#  I(a) + 1 \mod 2\Z. $
\end{prop}
Furthermore, if $(E,k)$ satisfies the Bohr-Sommerfeld condition,
then  
$$ \tilde{E}(E,k) = E - k^{-2}$$
and $\tilde{E}(E,k)$ is the joint eigenvalue associated to the coloring
$\varphi$ of $\Ga$ defined by $ \varphi ( a ) = k \la_a - 1$. 

\section{Scalar product of Lagrangian sections} \label{sec:scal-prod-lagr}

In this part we compute the asymptotics  of the scalar product of two
Lagrangian sections. The first section is devoted to algebraic
preliminaries.  
\subsection{A half-form pairing} \label{sec:pairing}
Consider  a symplectic vector space $(E^{2n}, \om)$ with a compatible
positive complex structure $J$. Given two transversal Lagrangian
subspaces $\Ga_1$ and $\Ga_2$, one has a sesquilinear non-degenerate
pairing
$$  \bigl( \wedge^n \Ga_1^*
\otimes \C \Bigr) \times  \bigl( \wedge^n \Ga_2^*
\otimes \C \Bigr) \rightarrow \C, \qquad \al, \be \rightarrow i^{n(2-n)}(\pi_1^* \al
\wedge \pi_2^*\overline{\be}) / \om^n
$$
where $\pi_1$ (resp. $\pi_2$) is the projection from $E$ onto $\Ga_1$
(resp. $\Ga_2$) with kernel $\Ga_2$ (resp. $\Ga_1$). Next for $i =1$ or $2$, the restriction from $E$ to $\Ga_i$ is
an isomorphism $$ \wedge^{n, 0} E^* \rightarrow \wedge^n \Ga_i^*
\otimes \C.$$ Composing these maps with the previous pairing, we obtain
a non-degenerate sesquilinear pairing
$$( \cdot, \cdot)_{\Ga_1, \Ga_2} :  \wedge^{n, 0} E^* \times \wedge^{ n,0} E^* \rightarrow \C
$$
We will consider a square root of this pairing. To remove the sign
ambiguity, let us compare it with the usual Hermitian product of the line $
\wedge^{n, 0} E^*$ 
$$ (\al , \be ) = i^{n ( 2 - n)} \al \wedge \overline{\be} / \om^n .$$ 
Recall that for any endomorphism $A$ of $J\Ga_1$ which is symmetric with respect to the scalar
product $\om (X, JY)$,  the graph of the map $JA: J \Ga_1 \rightarrow
\Ga_1$ is a Lagrangian subspace of $E = J \Ga_1 \oplus \Ga_1$. This defines a bijective correspondence between the set of Lagrangian
subspaces transversal to $\Ga_1$ and the affine space of symmetric
endomorphisms of $J\Ga_1$. One checks the following lemma with
a straightforward computation. 

\begin{lemma} \label{lem:form-pairing}
Let $A : J \Ga_1 \rightarrow J \Ga_1$ be the
  endomorphism   associated to $\Ga_2$. Then 
$$ ( \al , \be)_{\Ga_1, \Ga_2} = 2^{-n} \det ( 1 -i A) (\al , \be) $$ 
for any $\al, \be \in  \wedge^{n, 0} E^*$. 
\end{lemma} 

In particular the pairing induced by the couple $(\Ga_1, J\Ga_1)$ is
the same as the Hermitian product up to a factor $2^{-n}$.

Let us consider a one-dimensional Hermitian line $\delta$ with an isomorphism
$\varphi : \delta^{\otimes 2} \rightarrow \wedge ^{n, 0} E^*$. Then
one defines the sesquilinear pairing of $\delta$ 
by 
$$ ( \al, \be )_{\Ga_1, \Ga_2} = \sqrt{( \varphi(\al^2) , \varphi
 ( {\be} ^2) )_{\Ga_1, \Ga_2} }, \qquad  \forall \al \in \delta, \;
\forall \be \in {\delta}.$$
The square root is determined in such a way that the pairing depends
continuously on $\Ga_1$ and $\Ga_2$ and coincides with the scalar
product up to a factor $2^{-n/2}$  when $\Ga_2 = J \Ga_1$. Since the set of Lagrangian transversal to $\Ga_1$ is
affine, the square root is well-defined.  

\subsection{Asymptotics of scalar product of Lagrangian section} 

Let $(M^{2n}, \om)$ be a K{\"a}hler manifold with a prequantization
bundle $L$ and a half-form bundle $\delta$. Let $\Ga_1$ and $\Ga_2$ be
two Lagrangian submanifolds. For $i =1$ or
$2$, consider a sequence
$$ F_i^k a_i ( \cdot, k) \in \Ci (M, L^k \otimes \delta), \qquad k
=1,2,\ldots  $$
where 
\begin{itemize} 
\item $F_i$ is a section of $L$ such that its restriction at $\Ga_i$
is flat of norm $1$, $\overline \partial F_i$ vanishes at infinite order
along $\Ga_i$ and $|F_i( x)| < 1 $ if $x \notin \Ga_i$.  
\item $a_i(\cdot,k )$ is a sequence of $\Ci ( M, \delta)$ which admits
   full asymptotic expansion for the $\Ci$ topology 
$$ a_i ( \cdot , k ) = a_{i,0} + k^{-1}
  a_{i,1}  + k^{-2} a_{i,2} + \ldots $$ with coefficients $a_{i, \ell}$ in $\Ci (M, \delta)$. 
\end{itemize} 
Recall that the scalar product of two sections of
$L^k \otimes \delta $ is defined by
$$ ( \Psi_1, \Psi_2) = \int _M (\Psi_1 (x) , \Psi_2 (x) ) _{ L^k_x \otimes
  \delta_x} \mu_M (x) $$ 
where $\mu_M$ is the Liouville measure $\om^n/ n!$. 

\begin{theo} \label{theo:asympt-scal-prod}
Assume the intersection of $\Ga_1$ and $\Ga_2$ is transversal and
consists of a
  single point $y$, then
$$\Bigl( F_1^k a_1 ( \cdot, k), F_2^k a_2 ( \cdot, k) \Bigr)  \sim \Bigl( \frac{2 \pi}{k} \Bigr)^n \bigl( F_1 (y) ,
F_2 (y) \bigr)_{ L_y}^k \; \bigl( a_{1,0} (y) , a_{2,0} (y)
\bigr)_{T_y\Ga_1, T_y\Ga_2} $$
where $( \cdot, \cdot )_{ T_y\Ga_1, T_y\Ga_2}$ is the sesquilinear pairing
  $\delta_y \times \delta_y \rightarrow \C$ defined in section  \ref{sec:pairing}.
\end{theo} 

\begin{proof} Let us write
$ ( F_1 (x) , F_2 (x) )_{L_x}  = e ^{i \varphi (x) }$.  One has to
estimate
\begin{gather} \label{eq:int}
 \int_M e^{ik \varphi (x) } (a_1 (x, k) , a_2 (x, k) )_{ \delta_x}
\; \mu_M (x) 
\end{gather}
This will be an application of the stationnary phase lemma. First
since $|F_i( x)| < 1 $ if $x \notin \Ga_i$, the
imaginary part of $\varphi (x)$ is positive if $x \notin  \Ga_1 \cap \Ga_2$.
Next to compute
the derivatives of the function $\varphi$, let us recall the content
of lemma 4.2 of
\cite{oim_hf}.  One has  
\begin{gather} \label{eq:derF}
 \nabla^L F_i = \frac{1}{i} \al_i \otimes F_i
\end{gather}
where  $\al_i$ is a 1-form 
vanishing along $\Ga_i$. Furthermore for any vector fields $X$ and $Y$, one has
at $x \in \Ga_i$, 
\begin{gather} \label{eq:der2F}
 \dLie _X \langle \al_i, Y \rangle (x)  = \om ( q_i X(x), Y (x))
\end{gather}
where $q_i$ is the projection of $ T_x M \otimes \C$ onto
$T^{0,1}_x M$ with kernel $T_x \Ga_i \otimes \C$. 
One deduces from (\ref{eq:derF}) that 
$$ d \varphi =   \overline{\al}_2 -\al_1  $$
So the point $y \in \Ga_1 \cap \Ga_2$ is a critical point of $\varphi$ and one computes the Hessian at $y$ by using
(\ref{eq:der2F}):
\begin{xalignat}{2} \notag
 \operatorname{Hess} \varphi (X,Y) & = \dLie_X \bigl( (    \overline{\al}_2
 -\al_1 ) (Y) \bigr) \\  \notag  & = \om ( \overline{q}_2 X - q_1  X ,
 Y ) \\  \notag
 & = - i \om ( J (\overline{q}_2 + q_1)  X , Y ) \\ \label{eq:hess}
& = i g ( 
 (\overline{q}_2 + q_1)  X , Y ) 
\end{xalignat} 
where $g$ is the metric $g(X,Y) = \om (X, JY)$. At the third line, we
used that $\overline{q}_2 X \in T^{1, 0}_y M = \ker ( J_y - i )$
whereas $q_1 X \in T^{0,1}_y M = \ker ( J_y + i )$. Since the
intersection of $\Ga_1$ and $\Ga_2 $ is transverse, the Hessian is
non-degenerate. Then stationnary phase lemma (cf. \cite{Ho1} chapter
7.7) leads the following
equivalent  for the integral (\ref{eq:int})
$$ \Bigl( \frac{2 \pi}{k} \Bigr)^n e^{ik \varphi(y) }
\operatorname{det}^{-\frac{1}{2}} \bigl[ - i \operatorname{Hess}
\varphi (
\partial_{x_j}, \partial_{x_k} )(y) \bigr]_{j,k =1, \ldots , 2n}   \bigl( a_{1,0} (y) , a_{2,0} (y)
\bigr)_{\delta_y} \rho (y) 
$$
where $x_1, \ldots , x_{2n}$ are local coordinates at $y$ and $\rho (y)$ is such that $$\mu_M (y) = \rho (y) dx_1 \wedge \ldots
\wedge dx_{2n}.$$ 
Here the square root of the determinant is determined on the space of
symmetric complex matrices with a positive real part  in such a way
that it is positive on the subset of real matrices. The Liouville measure $\om^n
/n!$ is also the Riemannian measure of $g$, so that
$$ \rho (y) =  \operatorname{det}^{\frac{1}{2}} \bigl[ g ( 
\partial_{x_j}, \partial_{x_k} )(y) \bigr]_{j,k =1, \ldots , 2n} $$
Then it follows from (\ref{eq:hess}) that
$$ \operatorname{det} \bigl[ - i \operatorname{Hess}
\varphi (
\partial_{x_j}, \partial_{x_k} )(y) \bigr]_{j,k =1, \ldots , 2n}
\rho ^{-2}(y) = \det (\overline{q}_2 + q_1).$$
Finally the determinant of $ \overline{q}_2 + q_1$ is easily computed
in terms of the map $ A : J T_y \Ga_1 \rightarrow J T_y \Ga_1$ such
that $T_y \Ga_2$ is the graph of $JA : J T_y \Ga_1 \rightarrow T_y
\Ga_1$. One has
$$   \det (\overline{q}_2 + q_1) = \frac{2^n}{\det ( 1 - i A)} $$
Comparing with lemma \ref{lem:form-pairing}, we get the final result. 
\end{proof}

\section{Asymptotics of $6j$-symbols} \label{sec:asympt-6j-symb}

Assume that $n=4$ and choose
$(\ell_i)_{i =1, \ldots, 4}$ satisfying condition
(\ref{eq:hyp2}) and (\ref{eq:hyp1}). The moduli space 
$$\module_\ell =  (S^2_{\ell_1} \times S^2_{\ell_2} \times
S^2_{\ell_3} \times S^2_{\ell_4}) 
\varparallel SU(2) $$
is a 2-dimensional sphere. 
Let $\Ga$ be the left graph of figure \ref{fig:graph}.
\begin{figure} 
\begin{center}
\begin{picture}(0,0)%
\includegraphics{figure7.pstex}%
\end{picture}%
\setlength{\unitlength}{2921sp}%
\begingroup\makeatletter\ifx\SetFigFont\undefined%
\gdef\SetFigFont#1#2#3#4#5{%
  \reset@font\fontsize{#1}{#2pt}%
  \fontfamily{#3}\fontseries{#4}\fontshape{#5}%
  \selectfont}%
\fi\endgroup%
\begin{picture}(4604,2124)(2089,-5173)
\put(3751,-3811){\makebox(0,0)[lb]{\smash{{\SetFigFont{9}{10.8}{\rmdefault}{\mddefault}{\updefault}4}}}}
\put(3751,-4861){\makebox(0,0)[lb]{\smash{{\SetFigFont{9}{10.8}{\rmdefault}{\mddefault}{\updefault}3}}}}
\put(2401,-4861){\makebox(0,0)[lb]{\smash{{\SetFigFont{9}{10.8}{\rmdefault}{\mddefault}{\updefault}2}}}}
\put(2401,-3811){\makebox(0,0)[lb]{\smash{{\SetFigFont{9}{10.8}{\rmdefault}{\mddefault}{\updefault}1}}}}
\put(6601,-3361){\makebox(0,0)[lb]{\smash{{\SetFigFont{9}{10.8}{\rmdefault}{\mddefault}{\updefault}4}}}}
\put(6601,-5011){\makebox(0,0)[lb]{\smash{{\SetFigFont{9}{10.8}{\rmdefault}{\mddefault}{\updefault}3}}}}
\put(5251,-3361){\makebox(0,0)[lb]{\smash{{\SetFigFont{9}{10.8}{\rmdefault}{\mddefault}{\updefault}1}}}}
\put(5251,-5011){\makebox(0,0)[lb]{\smash{{\SetFigFont{9}{10.8}{\rmdefault}{\mddefault}{\updefault}2}}}}
\end{picture}%
\caption{The graphs $\Ga$ and $\Ga'$} \label{fig:graph}
\end{center}
\end{figure}
 Denote by
$H_{k \ell -1}$ the operator
associated to its internal edge and by  $h \in \Ci (\module_\ell)$ the
associated symbol
$$ h ( [u_1,u_2,u_3,u_4 ] ) = | u_1 + u_2 |^2, \qquad
(u_1,u_2,u_3,u_4) \in S^2_{\ell_1} \times \ldots \times S^2_{\ell_4}$$
Then $h(\module_\ell )= [m, M] $ where
$$ m = \max ( |\ell_1 - \ell_2 | , |\ell_3 - \ell_4 | ), \quad M =
\min ( \ell_1 + \ell_2 , \ell_3 + \ell_4 ) $$
The fibres of $h$ are circles except for the two singular ones
$h^{-1}(m)$ and $h^{-1} (M)$ which consist of one point. 

Consider now the right graph $\Ga'$ of figure \ref{fig:graph} and
denote by $H'_{k \ell-1}$ and $h'$ the
associated operator and symbol. Introduce two orthonormal bases $(\Psi_{E,k})_E$ and
$(\Psi'_{E,k})_{E }$ of eigenvectors of
$H_{k\ell-1}$ and $H'_{k \ell-1}$ respectively. 
$$  \frac{1}{k^2} H_{k \ell -1}  \Psi_{E,k} = E \Psi_{E,k} , \qquad 
  \frac{1}{k^2} H'_{k \ell -1}  \Psi'_{E,k} = E \Psi'_{E,k} $$
By the results of part \ref{sec:etats-propres} these eigenvectors are
Lagrangian sections associated to the level sets of $h$ and $h'$. From
this one
can deduce the asymptotics of the scalar product
$$ ( \Psi_{E, k}, \Psi '_{E',k}) $$
when $h^{-1} (E)$ and $h'^{-1} (E')$ intersect transversally, by
applying theorem \ref{theo:asympt-scal-prod}.

\subsection{The result}

Assume first that $h^{-1} (E_0)$ and $h'^{-1} (E_0')$ do not
intersect. Then it follows from theorem \ref{theorem:microsupport}
that there exist neighborhoods $V$ and $V'$ of $E_0$ and $E'_0$
respectively and a sequence $(C_N)$ such that 
$$ | ( \Psi_{E, k}, \Psi '_{E',k}) |   \leqslant C_N k^{-N}, \quad
\forall N,\; \forall k $$
and for any eigenvalues $E \in \Sp (k^{-2}H_{k \ell-1}) \cap V$ and $E' \in \Sp
(k^{-2} H'_{k \ell-1})
\cap V'$.

Next to
understand better the possible configurations of the level sets of $h$
and $h'$, it is useful to
think of $\module_\ell $ as a space of tetrahedra, the point
$[u_1, u_2, u_3, u_4 ]$ representing the tetrahedron $x$ of $\mathfrak{su} (2)$ with vertices $$0,
u_1, u_1 + u_2, u_1 + u_2 + u_3. $$
Furthermore two tetrahedra are identified if there are related by an
orientation-preserving isometry. The lengths of the edges of the
tetrahedron $x$ are 
$$ \sqrt
{h(x)}, \; \sqrt {h'(x)} ,\; \ell_1, \; \ell_2,\; \ell_3 \text{ and
}\ell_4 .$$ 
The subset $C$ of coplanar tetrahedra is an embedded circle. It contains the four points where $h$ and $h'$
attain their maximum and minimum.  At the other points of $C$, the
fibres of $h $ and $h' $ are tangent. Outside of this circle, the fibres
intersect transversally, cf. figure \ref{fig:fibration}.

\begin{figure}[htbp]
\begin{center}
    \includegraphics[width=0.6\linewidth]{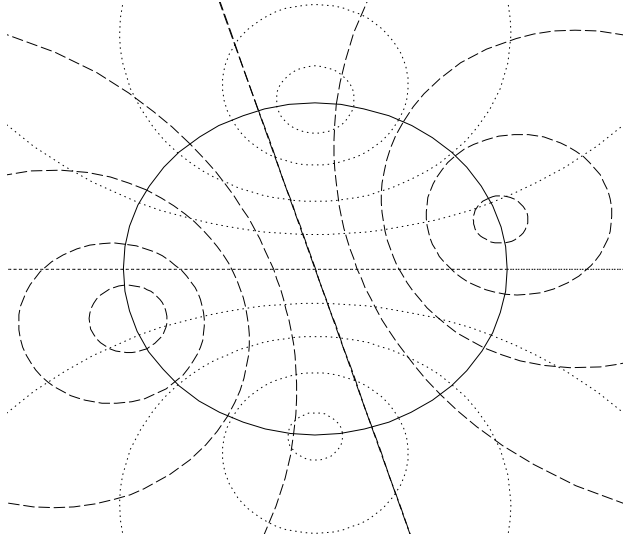}
    \caption{The two fibrations and the circle $C$}
    \label{fig:fibration}
  \end{center}
\end{figure}

Hence if $\sqrt{E}$, $\sqrt{E'}$, $\ell_1$, $\ell_2$, $\ell_3$ and 
$\ell_4$  are the edge lengths of a non-degenerate tetrahedron $\ta$, then $h^{-1} (E)$ and $h'^{-1} (E')$
intersect transversally at two points represented by $\ta$ and its
mirror image $\bar{\ta}$. Denote by $V(E,E')$ the volume of these
tetrahedra and by $\te(E, E')$ the sum
\begin{gather} \label{eq:deftetaEE}
\te(E,E') = \al  \sqrt{E} + \al' \sqrt{E'} +
  \sum_{i=1,\ldots ,4} \al_i \ell_i
\end{gather}
where  $\al$, $  \al'$, $\al_1$, $\al_2$, $\al_3$ and $\al_4$ are the
exterior dihedral angles (the exterior dihedral angle at an
edge is the angle in $[0,\pi]$ between the outward normal vectors of the
faces meeting at the edge).  

\begin{theo} \label{theo:asymptotism-6j}
Assume that $\sqrt{E_0}$, $\sqrt{E_0'}$, $ \ell_1$, $\ell_2$, $\ell_3$
and 
$\ell_4$ are the edge lengths of a non-degenerate tetrahedron. Then there exist neighborhoods $V$ and $V'$ of
$E_0$ and $E_0'$ respectively such that for every $k$ and for any
eigenvalues $E \in \Sp (k^{-2}H_{k \ell-1}) \cap V$ and $E' \in \Sp
(k^{-2} H'_{k \ell -1} )
\cap V'$, one has
$$  ( \Psi_{E, k}, \Psi '_{E',k}) = C_{k,E,E'} \sqrt{\frac{2}{3\pi}}
k^{-\frac{1}{2}} \frac{ (E E' )^{\frac{1}{4} }}{V(E,E')^{\frac{1}{2}}} \cos \bigl(k \te(E,E')/2  + \pi/4 \bigr) + O(k^{-\frac{3}{2}}) $$
where the $O$ is uniform with respect to $E$ and $E'$ and the
$C_{k,E,E'}$ are complex numbers of modulus 1.
\end{theo}

This result is proved in part \ref{sec:proof-6j-asymptotism}. The volume $V(E, E ')$ and the function
$\te (E, E ')$ appear in the computation of a holonomy and a
symplectic product, as it was already understood in \cite{WoTa}.
 
The $6j$-symbols are defined in terms of this scalar product by a
minor renormalisation 
$$ \left\{\begin{matrix} k \ell_1 -1 & k \ell_ 2 -1 & k \ell -1 \\  
 k \ell_3 -1 & k \ell_4 -1 & k \ell' -1 \end{matrix} \right\} = (-1)
^{ k( \ell_1 + \ell_2 + \ell_3 + \ell _4)/2}  \frac{(
\Psi_{E, k}, \Psi '_{E',k})}{ k \sqrt{ \ell \ell'}},
$$
with $E = \ell^2  -k^{-2}$ and  $E' = \ell'^2 - k^{-2}$. 
Here we assume that the bases $(\Psi_{E, k})$ and $(\Psi'_{E',k})$ are
suitably defined and not only up to a phase factor. If $\ell$, $\ell'$, $ \ell_1$, $\ell_2$, $\ell_3$
and $\ell_4$ are the edge lengths of a non
degenerate tetrahedron, we obtain
 $$ \left\{\begin{matrix} k \ell_1 -1 & k \ell_ 2 -1 & k \ell -1 \\  
 k \ell_3 -1 & k \ell_4 -1 & k \ell' -1 \end{matrix} \right\} \sim  \sqrt{\frac{2}{3\pi}}
k^{-\frac{3}{2}} V(\ell^2,\ell'^2)^{-\frac{1}{2}} \cos \bigl(k \te(\ell^2,\ell'^2)/2  + \pi/4 \bigr) $$
up to a phase factor, which is in agreement with the result of
Roberts in \cite{Ro}.

\subsection{Proof of theorem \ref{theo:asymptotism-6j}} \label{sec:proof-6j-asymptotism}

Consider a non-degenerate tetrahedron $\ta$ with edge lengths
$\sqrt{E}$, $\sqrt{ E '}$, $\ell_1$, $\ell_2$, $\ell_3$ and
$\ell_4$. Denote by $\bar \ta$ its mirror image. Then the
circle $h^{-1}(E)$ is the union of two segments delimited by $\ta$ and
$\bar {\ta}$. On one of these segments, $h'$ takes larger value than $E
' = h'(\ta) = h' (\bar{\ta})$. We shall denote it by $\ga$ and
orientate it according to the Hamiltonian
flow of $h$.  Consider in the same way the oriented segment
$\ga'\subset h'^{-1}(E')$. Then interchanging $\ta$ and $\bar{\ta}$
if necessary, one has $$\partial \ga = \bar{\ta} - \ta \quad \text{
  and } \quad \partial
\ga' = \ta - \bar{\ta}  .$$ 
Furthermore $\ga \cup \ga'$ divides the sphere $\module_\ell$ in two
domains. Let $D(E,E')$ be the one which does not contain $h^{-1}(E)
\setminus \ga$ and $h'^{-1}(E') \setminus \ga'$, cf. figure \ref{fig:domain}. We oriente it 
in such a way that the the symplectic area $\int_{D(E,E')} \om$ is
positive. The oriented boundary of $D(E, E')$ is then $-\ga \cup -\ga'$. 

\begin{figure}
\begin{center}
\begin{picture}(0,0)%
\includegraphics{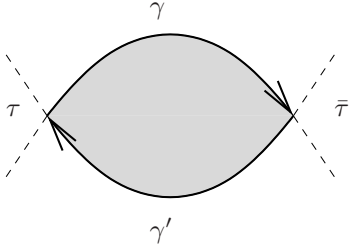}%
\end{picture}%
\setlength{\unitlength}{3355sp}%
\begingroup\makeatletter\ifx\SetFigFont\undefined%
\gdef\SetFigFont#1#2#3#4#5{%
  \reset@font\fontsize{#1}{#2pt}%
  \fontfamily{#3}\fontseries{#4}\fontshape{#5}%
  \selectfont}%
\fi\endgroup%
\begin{picture}(2424,1864)(1489,-1619)
\put(1501,-661){\makebox(0,0)[lb]{\smash{{\SetFigFont{10}{12.0}{\rmdefault}{\mddefault}{\updefault}$\ta$}}}}
\put(3901,-661){\makebox(0,0)[lb]{\smash{{\SetFigFont{10}{12.0}{\rmdefault}{\mddefault}{\updefault}$\bar{\ta}$}}}}
\put(2551, 89){\makebox(0,0)[lb]{\smash{{\SetFigFont{10}{12.0}{\rmdefault}{\mddefault}{\updefault}$\ga$}}}}
\put(2551,-1561){\makebox(0,0)[lb]{\smash{{\SetFigFont{10}{12.0}{\rmdefault}{\mddefault}{\updefault}$\ga'$}}}}
\end{picture}%
\caption{The domain $D(E,E')$} \label{fig:domain}
\end{center}
\end{figure} 
 Then adapting the results of \cite{WoTa} on  spherical tetrahedra to the
simpler Euclidean case, we prove the following 

\begin{prop} \label{prop:aire_angle}
The symplectic area of $D(E,E')$ is 
$$  A (E, E') = - \te(E,E') +  \pi ( \ell_1 + \ell_2 + \ell_3 + \ell_4 ) $$
and 
\begin{gather*}
 \om (X , X' ) |_{\ta} = - \om (X, X' ) |_{\bar{\ta}} =
3 \frac{V(E,E')}{\sqrt{E E'}}
\end{gather*}
where $X$ and $X'$ are the Hamiltonian vector fields of $\frac{1}{2}\sqrt{h}$
and $\frac{1}{2}\sqrt{h'}$ respectively and $\te(E,E')$ has been
defined in (\ref{eq:deftetaEE}). 
\end{prop}

By theorem \ref{theo:quasi_mode}, modifying the Lagrangian sections $\Psi_{E,k}$ and $\Psi'_{E',
  k'}$ by phase factors if necessary, one has on a neighborhood of
$D(E,E')$, 
\begin{gather*}
 \Psi_{E,k} = \Bigl( \frac{k}{2 \pi} \Bigr)^{\frac{1}{4}} F^k(E,
\cdot) g ( E,\cdot,k) + O( k^{-\infty}), \\ \Psi'_{E,k} = \Bigl( \frac{k}{2 \pi} \Bigr)^{\frac{1}{4}} F'^k(E,
\cdot) g' (E, \cdot,k) + O( k^{-\infty})
\end{gather*}
Then by theorem \ref{theo:asympt-scal-prod}, one has
\begin{gather} \label{eq:equiv}
 (\Psi_{E,k}, \Psi_{ E',k} ) \sim \Bigl( \frac{2 \pi}{k} 
\Bigr)^{\frac{1}{2}}\sum_{ x= \ta, \bar{\ta}} \bigl( F(E,x), F'(E',x)
\bigr)^k_{L_x} \bigl( g_0 (E,x), g_0'(E',x) \bigr)_{E,E', x} .
\end{gather}

Here $(\cdot, \cdot)_{E,E',x}$ is the  sesquilinear pairing of $\delta_x$
defined on part \ref{sec:pairing} with the Lagrangian subspaces $T_x h^{-1} (E)$ and
$T_x h'^{-1} (E ')$ of $T_x \module_\ell$. Denote by $\varphi$ the
isomorphism between $\delta^2_x$ and $\wedge^{1,0} T^*_x
\module_\ell$. By definition of the pairing one has 
\begin{gather} \label{eq:12}
 ( s, s' )_{E, E', x} ^2= 
i  \frac{ \varphi(s^2) (e  ) \; \overline{\varphi(s'^2) (e' )} }{ \om (e, e') }, \qquad \forall
\; s, s' \in \delta_x  
\end{gather}
for any nonvanishing vectors $e \in T_x h^{-1} (E)$ and
$e' \in T_x h'^{-1} (E ')$. 
It follows from condition 4. of theorem
\ref{theo:quasi_mode} that for $x = \ta$ or  $\bar{\ta}$, one has
\begin{gather} \label{eq:eqg0}
 \varphi ( g_0^2 (E,x) ) ( X(x) ) =  \varphi ( g_0'^2 (E',x) ) (
X'(x) ) = \frac{1}{2\pi}
\end{gather}
where $X$ and $X'$ are the Hamiltonian vector fields of $\frac{1}{2}\sqrt{h}$
and $\frac{1}{2}\sqrt{h'}$. Then we deduce from (\ref{eq:12}) that 
\begin{gather} \label{eq:pairing_hf} 
\bigl(g_0 (E,x), g_0'(E',x)\bigr)_{E,E', x} = \frac{1}{ 2 \pi}
\frac{e^{i \pi / 4} }{ \sqrt{\om  (X , X' )
   |_{x}} }
\end{gather}
 with the suitable determination of the
square root. So we deduce from the second part of proposition
\ref{prop:aire_angle} that both terms of the sum in equation
(\ref{eq:equiv}) has the
same modulus. We will prove that their phase difference is 
$$ k \bigl( \te (E, E ') -  \pi ( \ell_1 + \ell_2 + \ell_3 + \ell_4 )\bigr) + \pi/2
$$ 
Taking into account that $\ell_1 + \ell_2 + \ell_3 + \ell_4 $ is an even
integer, this will end the proof of theorem
\ref{theo:asymptotism-6j}.

 Since $F(E, \cdot)$ and $F'(E', \cdot)$ are flat along $\ga$ and
$\ga'$ respectively, one has
$$  \bigl( F(E,\bar \ta), F'(E',\bar \ta) \bigr)_{L_{\ta}} = H(E,E')\;
\bigl( F(E,\ta ), F'(E',\ta) \bigr)_{L_{\bar{\ta}}} $$
where $H(E,E') \in \C$ is the holonomy of $L$ along $\ga \cup \ga'$. Since $L$
has curvature $\frac{1}{i} \om$ and $-(\ga \cup \ga')$ is the
boundary of $D(E,E')$, one has 
$$H(E,E') = \exp (-iA (E,E'))$$
where $A(E , E ')$ is the symplectic area computed in proposition
\ref{prop:aire_angle}.  
Furthermore it follows from equation (\ref{eq:pairing_hf}) that 
$$ \bigl(g_0 (E,\bar{\ta}), g_0'(E',\bar{\ta})\bigr)_{E,E', \bar{\ta}}
= e^{\pm i  \pi/2}  \bigl(g_0 (E,\ta), g_0'(E',\ta)\bigr)_{E,E',
  \ta}$$
It happens that the undetermined sign is positive and the proof relies
uniquely on the configuration of the level sets, as they appear in
figure \ref{fig:domain}.

To see this, trivialise the tangent bundle of $\module_\ell$ on a neighborhood of $D(E,
E ')$ in such a way that the vector fields $X/ \sqrt {\om (X, JX)}$ and
  $JX / \sqrt {\om (X, JX)}$ are send to constant vectors that we denote
    by $e$ and $f$. The symplectic and complex structures
    are constant in this trivialisation
$$ \om (e, f) =1 , \quad Je = f $$
Trivialise also the half-form bundle, in such a way that the constant
half-form $s$ squares to $e ^* + i f ^*$, where $( e ^* , f^ *) $ is
the dual base of $(e,f)$. Then one may explicitely compute the
sesquilinear pairing associated to the Lagrangian lines generated
by $e$ and $f_\te = (\cos \te) f - (\sin \te )e$. By formula
(\ref{eq:12}) we have   
$$ (  s,  s )^2_{\R e, \R f_{\te}} =  \frac{e ^{-i \te }
}{\cos \te }$$
Choosing the determination of the square root as in part \ref{sec:pairing}, we obtain for $\te \in ( -\pi
/2, \pi /2 )$, 
\begin{gather} \label{eq:13}
 (  s,  s )_{\R e, \R f_{\te}} =  \frac{e ^{-i \te /2}
}{\sqrt{\cos \te }}
\end{gather}
Next introduce a parametrisation $x(t)$ of $\ga$ with $x(0) = \ta$
and $x(1) = \bar{ \ta}$. Then $$X (x (t)) = |X (x (t))  |\;  e $$
where $| X | = \sqrt{ \om (X, JX)}$. And modifying the sign of $g_0$ if necessary, one deduce from equation (\ref{eq:eqg0}) that
\begin{gather} \label{eq:14}
 g_0 ( E, x (t) ) =  \bigl(2 \pi | X (x(t))|\bigr)^{-\frac{1}{2}} \; s
\end{gather}

\begin{figure}
\begin{center}
\begin{picture}(0,0)%
\includegraphics{figure10.pstex}%
\end{picture}%
\setlength{\unitlength}{3158sp}%
\begingroup\makeatletter\ifx\SetFigFont\undefined%
\gdef\SetFigFont#1#2#3#4#5{%
  \reset@font\fontsize{#1}{#2pt}%
  \fontfamily{#3}\fontseries{#4}\fontshape{#5}%
  \selectfont}%
\fi\endgroup%
\begin{picture}(7216,2867)(1193,-8619)
\put(4501,-6061){\makebox(0,0)[lb]{\smash{{\SetFigFont{10}{12.0}{\rmdefault}{\mddefault}{\updefault}{\color[rgb]{0,0,0}$h^{-1}(E)$}%
}}}}
\put(5851,-7711){\makebox(0,0)[lb]{\smash{{\SetFigFont{10}{12.0}{\rmdefault}{\mddefault}{\updefault}{\color[rgb]{0,0,0}$\te(1/3)$}%
}}}}
\put(1951,-6361){\makebox(0,0)[lb]{\smash{{\SetFigFont{10}{12.0}{\rmdefault}{\mddefault}{\updefault}{\color[rgb]{0,0,0}$\te(1)$}%
}}}}
\put(2851,-6961){\makebox(0,0)[lb]{\smash{{\SetFigFont{10}{12.0}{\rmdefault}{\mddefault}{\updefault}{\color[rgb]{0,0,0}$\te(2/3)$}%
}}}}
\put(4501,-7861){\makebox(0,0)[lb]{\smash{{\SetFigFont{10}{12.0}{\rmdefault}{\mddefault}{\updefault}{\color[rgb]{0,0,0}$h'^{-1}(E')$}%
}}}}
\put(7201,-7111){\makebox(0,0)[lb]{\smash{{\SetFigFont{10}{12.0}{\rmdefault}{\mddefault}{\updefault}{\color[rgb]{0,0,0}$\te(0)$}%
}}}}
\end{picture}%
\caption{the angle $\te$} \label{fig:conf}
\end{center}
\end{figure}

Parametrize $-\ga'$ by $y (t)$ with $y(0) = \ta$
and $y(1) = \bar{ \ta}$. 
Then the configuration (cf. figure \ref{fig:conf}) of the level sets of
$h$ and $h'$ implies that 
$$ X'(y(t)) = r(t) \bigl( (\cos \te (t) ) e + (\sin \te (t)) f \bigr)
$$
where $r$ is a positive function and $ \te $ takes its values in $(0, 2
\pi)$ with $\te (0) \in ( 0, \pi )$ and $ \te (1) \in ( \pi, 2 \pi
)$. By equation (\ref{eq:eqg0})  
\begin{gather} \label{eq:15}
g_0' ( E ', y (t) ) =  \bigl(2 \pi r(t) \bigr)^{-\frac{1}{2}}  e ^{-i
  \te(t) /2} s
\end{gather}
Finally the angle between the lines generated by $X(\ta)$ and $X' (
\ta)$ being $\te(0) - \pi/2$, we deduce from equations (\ref{eq:13}),
(\ref{eq:14}) and (\ref{eq:15}) that 
\begin{xalignat*}{2}
 (g_0 (E, \ta), g'_0 ( E', \ta))_{E, E', \ta} =  &  \frac{e^{- i (
     \frac{\te(0)}{2} - \frac{\pi}{4} ) }}{ \sqrt {\cos ( \te (0) -
   \pi/2)} } . \frac{ e^{i \frac{\te (0)}{2}} }{ 2 \pi \sqrt {|X( \ta)| \; r (0) }}  
\\ = & \frac{e^{i \pi/4}
}{ 2 \pi } \Bigl( \cos ( \te (0) - \pi/2)  \;  | X (\ta) | \; |
X' ( \ta ) | \Bigr)^{-\frac{1}{2}}  
\end{xalignat*}
In the same way we obtain that 
\begin{xalignat*}{2}  
(g_0 (E, \bar \ta), g'_0 ( E',\bar  \ta))_{E, E', \bar \ta} = &  \frac{e^{- i (
     \frac{\te(1)}{2} - \frac{3\pi}{4} ) }}{ \sqrt {\cos ( \te (1) -
   3\pi/2)} } . \frac{ e^{i \frac{\te (1)}{2}} }{ 2 \pi \sqrt {|X( \ta)| \; r (1) }}  
\\ = & 
\frac{e^{i 3 \pi/4}
}{ 2 \pi } \Bigl( \cos ( \te (1) - 3\pi/2)  \;  | X (\bar \ta) | \; |
X' ( \bar \ta ) | \Bigr)^{-\frac{1}{2}}  
\end{xalignat*}
So the phase difference between the two pairings is $\pi /2$.

\section{Semi-classical reduction with subprincipal estimates} \label{sec:reduct-semi-class}

\subsection{Quantum reduction} \label{sec:data_reduction}

Let $(M, \om)$ be a connected compact K{\"a}hler manifold $(M, \om)$ with
a prequantization bundle 
$L \rightarrow M$ with curvature $\frac{1}{i} \om$ and a half-form
bundle $( \delta, \varphi)$. Here we denote by $\varphi$ the
line bundle isomorphism $\delta^2 \rightarrow \dcan T^* M$. By assumption it
preserves both the Hermitian and holomorphic structures. Let $G$ be a compact connected Lie group
acting on $M$ in a Hamiltonian way. Denote by $\mu : M
\rightarrow \Lie^*$ the moment map.
We assume that the action lifts to the prequantization bundle in such
a way that
the infinitesimal action on sections of $L$ is given by  
$$ \nabla_{\xi^\#} + i \mu^\xi,  
\quad \forall \xi \in \Lie .$$
We assume furthermore that the action preserves the complex structure
and lifts to the half-form bundle in such a way that $\varphi$ is equivariant.
Under these assumptions, the group $G$ acts naturally on the space
$H^0 (M , L^k \otimes \delta) $ for any positive integer $k$,  the
infinitesimal action being given by the Kostant-Souriau operators
(\ref{eq:KS}). We denote by $$H_G^0 ( M, L^k  \otimes \delta)$$ the
$G$-invariant subspace. 

Suppose that $G$ acts freely on the zero-set $P := \mu^{-1}
(0)$. Then $0$ is a regular value of the moment, $P$ is a coisotropic
submanifold of $M$ and its characteristic distribution is the tangent
space to the orbits. So the quotient $$M_r
:= P / G$$ is a symplectic manifold. Consider the quotient $L_r$ of
the restriction of $L $ to $P$ by the $G$-action. Since the action preserves the
connection of $L$ and is by parallel transport over $P$, $L_r$ inherits a connection. Its curvature is
$\frac{1}{i} \om_r$ where $\om_r$ is the reduced symplectic form. 

To define the complex structure on the symplectic quotient, introduce
the complexification $G^{\C}$ of $G$. It is a complex
connected Lie group containing $G$ as a maximal compact subgroup. The
Lie algebra of $G^{\C}$  is the complexification of $ \Lie$ and the
Cartan decomposition is the diffeomorphism $$G^\C \simeq \exp ( i
\Lie ) G .$$ 
Furthermore the set $\exp ( i
\Lie )$ is diffeomorphic to the vector space $ \Lie$, the
diffeomorphism being the exponential map. 

The $G$-action extends to a holomorphic action of  $G^\C$
whose infinitesimal action is given by  
$$ \xi^\# + J \eta^\# ,\qquad  \forall \xi + i \eta \in \Lie \oplus i \Lie =
\Lie\otimes \C .$$
The saturated set $M_s := G^\C. P$ of the zero set of $\mu$ is called the stable set.  
Since the vector field $ J \xi^\#$ is the Riemannian gradient of 
$\mu^\xi$ for any vector $\xi$, $M_s$ is an open set diffeomorphic to $  \Lie \times P$,
the diffeomorphism being 
\begin{gather} \label{eq:1} 
  \Lie \times P \rightarrow M_s, \quad (\xi , x ) \rightarrow \exp ( i
\xi ) .x
\end{gather}
Furthermore the action of $G$ on $P$ being free, the action of $G^\C$
on $M_s$ is also free. Finally the injection of $P$ into $M$ induces a
diffeomorphism 
$$P / G \simeq M_s / G^\C.$$ In this way the symplectic quotient
inherits a complex structure. It is compatible with the symplectic
form and $M_r$ becomes a K{\"a}hler manifold. 

Similarly the $G$-action on the prequantization bundle and the
half-form bundle can be analytically continued to holomorphic actions
of the complexified group $G^\C$.  We have a natural identification
between $L_r$ and the quotient by $G^\C$ of the restriction of $L$ to
the stable set. Hence $L_r$ inherits a holomorphic structure, it is compatible
with the connection. We define $\delta_r$ as the quotient by $G^\C$ of
the restriction of $\delta$ to $M_s$. This is a holomorphic line bundle
on $M_r$.
The holomorphic $G$-invariant sections of $L^k \otimes \delta$ are also
invariant under the complexified action. This defines a natural map 
\begin{gather} \label{eq:GS_def}
 V_k : H_G^0 (M, L^k \otimes \delta ) \rightarrow H^0 (M_r, L_r^k
\otimes \delta_r) .
\end{gather}

\begin{theo}[Guillemin-Sternberg] \label{sec:GS_isomorphism} When $k$
  is sufficiently large, $V_k$ is an isomorphism.
\end{theo}

The condition on $k$ is due to the presence of the half-form
bundle. It does not appear in theorem \ref{theo:geom-real-space} because in this case the
half-form bundle is a power of the prequantum bundle.  

The theorem says that the $G^\C$-equivariant holomorphic sections of $L^k
\otimes \delta$ over the stable set extend uniquely into $G^\C$-equivariant
holomorphic sections over $M$. This follows from the non-trivial fact
that the complementary of the stable set is contained in a complex
submanifold of codimension $\geqslant 1$. Furthermore  the $G^\C$-equivariant sections of $L^k
\otimes \delta \rightarrow M_s$ are bounded when $k$ is sufficiently
large.  In the next section we will prove an explicit estimate that we
will use in the sequel. 

\subsection{Estimates of the equivariant sections} 

Introduce a norm $\| . \|$ on the Lie algebra $\Lie$.
\begin{prop} \label{sec:estimation}
There exists $C_1, C_2 >0$ such that for any integer $k$ and any $G^\C$-equivariant section
$\Psi$ of $L^k \otimes \delta \rightarrow M_s$, one has
$$ \| \Psi( \exp(i \xi) .x) \|^2 \leqslant e^{ C_1 \| \xi \| - k
  C_2 \| \xi \|^2} \| \Psi ( x)\|^2, \quad \forall x \in P,\; \forall \xi
\in \Lie  $$
where we denote by $\| \cdot \|$ the punctual norm of $L ^k \otimes
\delta$. 
\end{prop}

\begin{proof} We start by estimating $\| \exp ( i\xi). u\|$
  in terms of $\| u\|$ for any $u \in \delta$.  
Introduce the smooth function $r$ on $\Lie \times M$
$$ r(\xi, x) := \frac{d}{dt} \Bigr|_{t=0} \frac{ \| \exp (i t \xi).u
  \| ^2}{\| u \| ^2} , \qquad \text{ where } u \in \delta_x $$
We have 
$$ \frac{d}{dt} \| \exp (i t \xi).u \|^2 =  r(\xi, \exp( it \xi ). x)
\| \exp ( it \xi) .u \| ^2 $$
Let $C_1$ be the supremum of $ | r |$ on the compact set $ \{ \| \xi
\| =1 \} \times M$. Then by integrating the previous equality we get for $\| \xi \| =1 $ that 
$$ \| \exp (i t \xi) .u\|^2 \leqslant e^{C_1 | t| } \| u \|^2 .$$
Let us now estimate $\| \exp ( i\xi). u\|$
  in terms of $\| u\|$ for any $u \in L$. First we have 
$$  \frac{d}{dt} \Bigr|_{t=0} \frac{ \| \exp (i t \xi).u
  \| ^2}{\| u \| ^2} =  2 \mu^\xi(x)  , \qquad \text{ if } u \in L_x $$
This follows from the fact that the infinitesimal action of $i \xi \in \Lie
\otimes \C$ on the sections of $L$ is $\nabla_{J \xi^\#} -
\mu^{\xi}$. Furthermore since $J \xi ^ \#$ is the gradient of
$\mu^\xi$, we have  
$$ \frac{d}{dt} \mu^\xi (\exp (i t \xi).x ) =  - g ( \xi^\#,
\xi^\# ) (\exp (i t \xi).x)$$
where $g$  is the Riemannian metric $g(X, Y) = \om (X, JY)$ of $M$. 
Let $C_2$ be the infimum of $g(\xi ^\#, \xi ^\# ) (x) $ on the compact $\{ \| \xi
\| =1 \} \times M$. By integrating, we obtain for  $\| \xi \| =1$ 
$$  \mu^\xi (\exp (i t \xi).x ) \leqslant \mu^\xi (x) - C_2 t$$  
Integrating again we obtain for $x\in P$ (and consequently $\mu^\xi (x)
=0$) that  
$$ \| \exp (i t \xi) .u\|^2 \leqslant e^{-C_2  t^2 } \| u \|^2 ,
\qquad \forall \; u \in L_x$$
This ends the proof.\end{proof}

\subsection{The reduced half-form bundle} 
In this part we define the isomorphism $\varphi_r : \delta^2_r
\rightarrow \dcan T^*M_r$ which makes $\delta_r$ a half-form
bundle. Recall that the isomorphism $\delta \rightarrow \dcan T^*M$
was denoted by $\varphi$. 
 Let $\pi_s$ be the projection from the stable set $M_s$ onto
the quotient $M_r = M_s /
G^\C$. Introduce the bundle over $M_s$ 
$$ E =  (\ker (\pi_s)_\star \otimes \C ) \cap T^{1,0} M_s .$$
One has an exact sequence 
$$ 0 \rightarrow  E
\rightarrow T^{1,0} M_s \rightarrow \pi_s ^* T^{1,0} M_r \rightarrow
0, 
$$
Consider an invariant metric $( ., . )$ on the Lie algebra $\Lie$ and an orthonormal
base $(\xi_i)$ of $\Lie$. Define the section of $\dvol E$
\begin{gather} \label{eq:7} 
 \gamma =   \frac{1}{2^\ell}( \xi_1^\# - i J \xi_1 ^\# ) \wedge
\ldots \wedge ( \xi_\ell^\# - i J \xi_\ell ^\# ) 
\end{gather}
This section $\ga$ does not depend on the choice of the base $(
\xi_i)$. It does not vanish, is holomorphic and
$G^\C$-equivariant. The contraction by $\ga$ defines an
$G^\C$-equivariant isomorphism
$$   \dcan T^* M_s \rightarrow \pi_s ^* \dcan T^* M_r 
$$
By composing  with $\varphi$, we obtain an isomorphism from
$\delta^{\otimes 2}$ to $\pi_s ^* \dcan T^* M_r$ which descends into
$$ \varphi_r : \delta_r ^{\otimes  2} \rightarrow \dcan T^* M_r $$
In other words, for any $u \in  \delta$, one has 
\begin{gather} \label{eq:8} 
 \iota (\gamma) \varphi (u ^2 ) = \pi_s^* \varphi_r ( [ u ]^2 ). 
\end{gather}
Observe that $\varphi_r$ is a holomorphic map. 

We endow $\delta_r$ with the
metric such that $\varphi_r$ become an isomorphism of
Hermitian bundle. We have to be careful that the projection from
$\delta |_{M_s} $ onto $\delta_r$ does not preserve the metrics, even when it is
restricted to the zero set of the moment map. For any  $u, v \in
\delta_x$ with $x$ in the zero level set,
\begin{gather} \label{eq:6}
   (u,v )_{ \delta_x}  = \| \gamma (x) \|^{-1}  \; (  [u], [v])_{\delta_{r,x} }  
\end{gather}
And a straightforward computation gives the punctual norm of $\ga$  
$$ \|  \ga (x)  \| = 2^{-\ell/2}  \operatorname{det} ^{\frac{1}{2}}\bigl[ g( \xi_i^\# , \xi_j ^\# )
\bigr]_{i,j=1, \ldots , \ell} (x) $$
where $ g$ is the metric $\om (X, JY)$. From now on, we assume that
the invariant metric of $\Lie$ is chosen so that the Riemannian volume
of $G$ is 1. This implies that the Guillemin-Sternberg
isomorphism rescaled by a factor $ ( 2\pi/k)
^{\ell /4  }$ is
asymptotically unitary as was proved in \cite{HaKi}. This follows also
from the next results, cf. the remark after theorem \ref{sec:compF}.

\subsection{A class of Fourier integral operators} \label{sec:four-integr}

We denote by $\pi$ the projection from the zero set $P$ of the moment map 
onto the quotient $M_r= P /G$. Let
us introduce some datas associated to the symplectic reduction $M
\varparallel G$. First consider the submanifold  
$$ \La := \{ (x, \pi (x) ), x\in P \} \subset M \times M_r
$$ 
Denote by $M_r^-$ the manifold $M_r$ endowed with the symplectic form
$- \om_r$. Then $\La$ is a Lagrangian submanifold of $M \times
M_r^-$ called the {\em moment Lagrangian}. Next one defines a section $t$ of $L \boxtimes
\overline{L}_r$ over $\La$ by   
 $$ t ( x, \pi (x) ) = u \otimes [\overline{u}] $$ 
if $x \in P$ and  $u \in L_x $ is a unitary vector. This section is
flat and unitary. In a similar way
consider the section $t_\delta$ of $\delta  \boxtimes
\bar{\delta}_r$ defined on $\La $ by  
\begin{gather} \label{eq:16}
 t_{\delta} (x,\pi(x)) = u \otimes [\overline{u}]
\end{gather}
 if $x \in P$, $ u \in \delta _{x}$ and $[u] \in \delta_{r,x}$ is a
unitary vector.

Next we define a space $\Fourier$ of Fourier integral operators
associated to $( \La, t, t_\delta)$. First let us introduce the
Schwartz kernel of an operator $ \Hilbert_{r, k} \rightarrow
\Hilbert_k$ where 
$$ \Hilbert _k := H^{0} ( M, L^k \otimes \delta) , \qquad \Hilbert _{r,k} := H^{0} ( M_r, L_r^k \otimes \delta_r).$$
  The scalar product of $\Hilbert _{r,k}
$ gives us an isomorphism 
$$ \End (\Hilbert _{r,k} , \Hilbert _k  ) \simeq  \Hilbert _k 
\otimes \overline{\Hilbert}_{r,k} .$$
The latter space can be regarded as the space of holomorphic sections of  
$$ (L^k \otimes \delta ) \boxtimes (\bar{L}^k_r \otimes \overline{\delta}_r)
\rightarrow M\times \overline{M}_r,$$
where $\overline{M}_r$ is the manifold $M_r$ endowed with the conjugate
complex structure. The section associated in this way to an operator is its
Schwartz kernel. 

 Consider a sequence
$(Q_k)$ such that for every $k$, $Q_k$ is an operator $\Hilbert_{r,k}
 \rightarrow \Hilbert _{k} $.
We
say that $(Q_k)$
is a Fourier integral operator of $\Fourier $ if the sequence of
Schwartz kernel satisfies
\begin{gather} \label{noyau_FIO}
 Q_k(x,y) =  \Bigl( \frac{k}{2\pi} \Bigr)^{n_r+ \frac{\ell}{4}} F^k(x,y) f(x,y,k) + O
(k^{-\infty}) \end{gather}
where 
\begin{itemize}
\item 
$F$ is a section of  $L \boxtimes \bar{L}_r
\rightarrow M\times \bar{M}_r$ such that  $\| F(x,z)
\| <1 $ if $(x,z) \notin \La  $, 
$$ F (x,z) = t (x,z), \quad \forall (x,z) \in \La $$
and $ \bar{\partial} F \equiv 0 $
modulo a section vanishing to any order along $\La$.
\item
  $f(.,k)$ is a sequence of sections of  $ \delta  \boxtimes  \bar{\delta}_r
\rightarrow M\times \bar{M}_r$ which admits an asymptotic expansion in the $\Ci$
  topology of the form 
$$ f(.,k) = f_0 + k^{-1} f_1 + k^{-2} f_2 + ...$$
whose coefficients satisfy $\bar{\partial} f_i \equiv 0  $
modulo a section vanishing to any order along $\La$.
\end{itemize}
Furthermore $n_r$ is the complex dimension of $M_r$ and $\ell$ is the
dimension of $G$. 

Let us define the principal symbol of $(Q_k)$ to be the function $g
\in \Ci (P)$ such that the restriction to $\La$ of the leading term $f_0$ of the previous asymptotic
expansion is 
$$  f_0 (x,z) = g (x) t_{ \delta} (x,z) , \qquad \forall (x,z) \in
\La $$
Denote by $\sigma : \Fourier  \rightarrow \Ci(P)$ the principal symbol
map. Let $\Fourier^1$ be the space of Fourier integral operator of
order 1, that is $(T_k) \in \Fourier^1 $ if and only if $(k
T_k) \in \Fourier$.  
The next result  was proved in \cite{oim_qm}.
\begin{theo} 
The sequence 
$ 0 \rightarrow  \Fourier^1 \rightarrow 
\Fourier  \xrightarrow{\sigma} \Ci(P) \rightarrow 0 $ 
is exact. 
\end{theo}
\subsection{Semiclassical properties of the reduction} 

Introduce the inverse  of the
Guillemin-Sternberg isomorphism (\ref{eq:GS_def})  
$$ W_k : H^0(M_r, L^k_r \otimes \delta_r) \rightarrow H^0_G(M, L^k
\otimes \delta ) .$$ 
Recall that we denote by $\ell$ the dimension of $G$.
\begin{theo} \label{sec:isoGS}
The sequence $\bigl ( \bigl( \frac{k}{2\pi} \bigr) ^{\frac{\ell}{4}}
W_k \bigr) $ is an Fourier integral operator of $\Fourier$ whose principal symbol is
the constant function equal to 1. 
\end{theo}

We defined Toeplitz operator in section
\ref{sec:toeplitz-operators} with their principal and subprincipal
symbols. In the following theorems we consider Toeplitz operators on
$M$ 
and the reduced space $M_r$.
\begin{theo} \label{sec:compF}
Let $(Q_k)$ and $(R_k)$ be Fourier integral of $ \Fourier$ with symbol
$f_Q$ and $f_R$ respectively. Then $(Q_k^* R_k)$ is a Toeplitz
operator of $M_r$. Its principal symbol is the function $g \in
\Ci (M_r)$ given by 
$$g (z) = \int _{P_z} \overline{f_Q (x)} f_R (x) \; \mu_{P_z} (x) , \qquad
\forall z \in M_r$$
where $\mu_{P_z} $ is the $G$-invariant measure of $P_z=\pi^{-1}(z)$
normalized by $\int _{P_z} \mu_{P_z}=1 $. 
\end{theo}

As a consequence of both theorems, $\bigl ( \bigl( \frac{k}{2\pi} \bigr) ^{\frac{\ell}{2}}
W_k^* W_k  \bigr) $ is a Toeplitz operator with principal symbol
1. Hence its uniform norm is equivalent to 1 as $k$ tends to
$\infty$. In other words, the Guillemin-Sternberg
isomorphism rescaled by a factor $ ( 2\pi/k)
^{\ell /4  }$ is
asymptotically unitary, as it was already proved in \cite{HaKi}.   
  Last result is about the composition of Toeplitz operators with
Fourier integral operators. 
 
\begin{theo} \label{sec:sous-princ}
Let $(Q_k)$ be a Fourier integral operator of $\Fourier$. Let $(S_k)$
and $(T_k)$ be Toeplitz operators of $M$ and $M_r$
respectively with principal symbols
$f_0$ and $g_0$. Then the sequence $(S_kQ_kT_k)$ is a
Fourier integral operator of $\Fourier$ with symbol  
$$  (j^* f_0 ) \si (Q) (\pi^* g_0) $$ 
where $j$ is the injection $ P \rightarrow M$
and $\pi$ the projection $P \rightarrow M_r$. Assume furthermore that $f_0$ is
$G$-invariant and $ j^* f_0 = \pi^* g_0$, then $$S_k Q_k -
Q_kT_k = k^{-1} R_k $$ 
where $(R_k)$ is a Fourier integral operator of $\Fourier$ with symbol 
$$ \si ( R_k) = ( j^* f_1 - \pi^* g_1 + \tfrac{1}{i} \dLie_X ) \si (Q)$$
where $f_1$ and $g_1$ are the subprincipal symbols of $(S_k)$ and
$(T_k)$ respectively and $X$ is the restriction to $P$ of the
Hamiltonian vector field of $f_0$. 
\end{theo}

The previous theorem are proved in the following sections. Before let us deduce the result about the reduction of Toeplitz
operators. Introduce the isomorphism of Hilbert spaces: 
$$ U_k = W_k ( W_k^* W _k)^{-\frac{1}{2}} :  H^0(M_r, L^k_r \otimes \delta_r) \rightarrow H^0_G(M, L^k
\otimes \delta ) . $$ 
As already noted, $ \bigl( ( \frac{k }{2 \pi} )^{\ell / 2} W_k^* W_k\bigr)$ is a
Toeplitz operator with principal symbol the constant function equal to
1. So the same holds for $( ( \frac{k }{2 \pi}
)^{-\ell/4} (W_k^* W_k\bigr)^{-1/ 2})$.   
Then theorem \ref{sec:sous-princ} implies that $(U_k)$ is a Fourier integral operator of
$\Fourier$ with symbol the constant function equal to 1.

\begin{cor} \label{the_corollaire}
Let $(S_k)$ be a Toeplitz operator of $M$ with principal symbol
$f_0$. Then  $(U^*_k  S_k U_k)$ is a Toeplitz operator of $M_r$ whose
principal symbol $g_0$ is given by  
$$ g_0 (z) = \int_{P_z} f_0 (x) \mu_{P_z} (x) $$
If furthermore $f_0$ is $G$-invariant, then the subprincipal symbol
is 
$$ g_1 (z) = \int_{P_z} f_1 (x) \mu_{P_z} (x) $$
where $f_1$ is the subprincipal symbol of $(S_k)$. 
\end{cor}

The computation of the principal symbol is an immediate
consequence of the previous theorems. To compute the subprincipal term,
introduce a Toeplitz operator $(T_k)$ of $M_r$ whose principal symbol is the function $g_0$ defined by the
integral in the statement of the corollary. Then by theorem
\ref{sec:sous-princ}, one has $ S_k U_k - U_kT_k = k^{-1}R_k$ where
$(R_k)$ is a Fourier integral operator with symbol $ i^* f_1 - \pi^*
g_1$ and $g_1$ is the subprincipal symbol of $(T_k)$. Assume
that $g_1$ is equal to the second integral in statement of the corollary so that the symbol of
$(R_k)$ vanishes. Composing with $U_k^*$, it comes that   $$U_k^* S_k U_k   - T_k  =
k^{-1}U_k^* R_k  . $$ 
By theorem \ref{sec:compF}, $(U_k^* R_k )$ is a Toeplitz operator with
vanishing principal symbol. Hence $(U_k^* S_k U_k)$ and $(T_k )$ have
the same principal and subprincipal symbols.

\subsection{Proof of theorem \ref{sec:isoGS}}

The proof relies on some properties of Toeplitz operators proved in
\cite{oim_bt} that we recall now. First the Schwartz kernel of a Toeplitz
operator is described in a similarly way as the
one of a Fourier integral operators in section \ref{sec:four-integr}.  The Schwartz kernel
of a Toeplitz operator $(T_k)$ of $M_r$ is a sequence of 
holomorphic sections of  
$$ (L_r^k \otimes \delta_r ) \boxtimes (\bar{L}^k_r \otimes \overline{\delta}_r)
\rightarrow M_r \times \overline{M}_r,$$
of the form 
\begin{gather} \label{noyau_toep} 
 \Bigl( \frac{k}{2\pi} \Bigr) ^{n_r}   E_r^k (x,y) f (x,y,k)
+ O( k^{-\infty}). 
\end{gather}  
where $n_r$ is the complex dimension of $M_r$, $E_r$ and $ f( \cdot, k)$ satisfy similar assumptions to those
of section \ref{sec:four-integr} with the moment Lagrangian replaced
by the diagonal. Precisely, $E_r$ is a section of $L_r
 \boxtimes \bar{L}_r \rightarrow M_r \times \overline{M}_r$ whose
 restriction to the diagonal satisfies
$$ E_r (x,x) = u \otimes \bar u$$
for any unitary vector $u \in L_{r,x}$. Furthermore $\bar{\partial}
E_r$ vanishes to any order along the diagonal and outside the diagonal
one has  $ \| E_r (x,y ) \| < 1$. $(f ( \cdot
, k))$ is a sequence of sections of  $ \delta_r  \boxtimes  \overline{\delta}_r
\rightarrow M_r \times \overline{M}_r$ which admits an asymptotic expansion in the $\Ci$
  topology of the form 
$$ f(.,k) = f_0 + k^{-1} f_1 + k^{-2} f_2 + ...$$
whose coefficients satisfy $\bar{\partial} f_i \equiv 0  $
modulo a section vanishing to any order along the diagonal.
Finally we recover the principal symbol $h$ of the operator $(T_k)$
from the first coefficient $f_0$ by 
$$ f_0(x,x) = h(x) t_r(x,x) $$
Here $t_r$ is the section of $\delta_r \boxtimes \overline{\delta}_r$
over the diagonal of  $M_r^2$ such that 
\begin{gather}\label{eq:deftr}  
t_r(x,x) = u \boxtimes \overline u
\end{gather}
if  $u$ is  a unitary vector of  $\delta_{r,x}$.

\begin{lemma} Let $(T_k)$ be a Toeplitz operator of $M_r$ with symbol
  $h$. Then  
 $$\Bigl( \frac{k}{2\pi} \Bigr) ^{\frac{\ell}{4}} W_k T_k$$ is a
 Fourier integral operator of  $\Fourier$ with symbol $\pi^* h$, where
  $\pi$ is the projection $P \rightarrow M_r$.
\end{lemma}

Theorem \ref{sec:isoGS} follows.
\begin{proof}  Recall that we denote by $\pi_s$ the projection from the stable
  set $M_s$ onto $M_r$. By definition of the Guillemin-Sternberg
  isomorphism, the Schwartz kernel of $W_k  T_k$ over $M_s \times M$
is the pull-back of the Schwartz kernel of $T_k$ by the projection $\pi_s \times
\id$. The result follows by comparing (\ref{noyau_FIO}) and
(\ref{noyau_toep}). We deduce from proposition \ref{sec:estimation},
that the Schwartz kernel of  $W_k 
T_k$ and its successive derivatives are $O( k^{-\infty})$ on the
complementary set of $ \La$. Next observe that we can choose the
sections $E_r$ and $F$ in such a way that  
\begin{gather} \label{eq:2'}
(\pi_s \boxtimes \id )^* E_r = F
\end{gather}
and the final result follows easily. 
\end{proof} 

\subsection{Proof of theorem \ref{sec:compF}}

The Schwartz kernel of $Q_k^* R_k$ is given by the following integral 
\begin{gather} \label{eq:3}
 \Bigl( \frac{k}{2 \pi} \Bigr)^{2n_r + \frac{\ell}{2}}  \int_{M} \overline{F}^k (y, z_1 ) .F^k(y,z_2) f(y,z_1,z_2,k) \; \mu_M
(y)
\end{gather}
where $f( \cdot , k)$ admits an asymptotic expansion in inverse power
of $k$. The leading order satisfies
$$ f(x,z,z,k) =  \overline{f_Q (x)} f_R (x) \; \overline{t_{\delta}
(x,z)}. t_{\delta} (x,z) +O(k^{-1}) $$
for any $(x,z)$ in the moment Lagrangian, where $f_Q$ and $f_R$ are the symbols of $(Q_k)$ and $(R_k)$ 
respectively. We deduce from (\ref{eq:6}), (\ref{eq:16}) and
(\ref{eq:deftr}) that 
\begin{gather} \label{eq:17}
 f(x,z,z,k) =  \overline{f_Q (x)} f_R (x) 2^{\frac{\ell}{2}} \operatorname{det}
 ^{-\frac{1}{2}} [ g ( \xi_i^\#, \xi _j^ \#) ] (x) \; t_r (z,z) +
 O(k^{-1} )
\end{gather}
Let us compute (\ref{eq:3}). First, since the punctual norm of $F$ is smaller than $1$ outside $\La$, 
by modifying the integral by a $O(k^{-\infty})$ if necessary, we may assume
that the support of $f$ is a compact of $M_r \times M_s \times M_r$. 
So we just integrate on the stable set. 
Let us identify $M_s$  with $\Lie \times P$ by the
diffeomorphism (\ref{eq:1}). We will integrate
successively on $\Lie$, then on the fiber of $\pi: P
\rightarrow M_r$ and finally on $M_r$. To do that, we write the
Liouville measure on the stable set in the following way. 
\begin{lemma} \label{sec:lemme_mesure}
We have over $\Lie \times P$ 
$$ \mu_M(\xi , x) = \delta (\xi , x)  \mu_P(x) | dt_1 ...d t_\ell | (\xi) $$
where $\mu_P$  is the invariant measure of $P$ whose push-forward by
the projection $P \rightarrow M_r$ is the Liouville measure $\mu_{M_r}$ 
of $M_r$, $t_1, \ldots, t_\ell$ are linear coordinates of $\Lie$
in a orthogonal base $\xi_1, \ldots , \xi_\ell$ and $\delta \in
\Ci ( \Lie \times P)$ satisfies
$$ \delta ( 0, x) = \det [ g ( \xi_i^\#, \xi _j^ \#) ] (x), \qquad
\forall x\in P . $$ 
\end{lemma}
This follows from the fact that the Riemannian volume of $G$ is 1
(cf. lemma 4.20 of \cite{oim_qr} for a proof).
Furthermore we deduce from (\ref{eq:2'}) that for $y = \exp (i\xi ). x$ with $x \in
P$, one has 
$$  \overline{F}^k (y, z_1 ). F^k(y,z_2) =  e^{-k \varphi (\xi,x)}
\overline{F}^k (x, z_1 ). F^k(x,z_2)
$$
with
$$ \varphi (\xi,x )= - \ln \Bigl( \frac{ \| \exp ( i \xi ) . u \| ^2
}{\| u \| ^2 } \Bigr), \qquad \text{ if } u \in L_x .$$
By the Fubini theorem, the integral (\ref{eq:3}) is equal to  
\begin{gather} \label{eq:4}
  \Bigl( \frac{k}{2 \pi} \Bigr)^{2n_r }  \int_{M_r} \overline{E}_r^k (z, z_1 ) .E_r^k(z,z_2)  f'' (z,z_1,z_2,k)  \; \mu_{M_r} (z)
\end{gather}
where $f''$ is the function of $M_r^3$ given by 
$$  f'' (z,z_1,z_2,k)   = \int_{P_z} f'(x,z_1,z_2) \;
\mu_{P_z} (x)$$
with  $\mu_{P_z}$ the $G$-invariant measure of $P_z$ with total volume
1. And  $f'$ is the function of $P \times M_r^2$ given by 
\begin{gather*} 
 f'(x,z_1,z_2,k) =  \Bigl( \frac{k}{2 \pi} \Bigr)^{\frac{\ell}{2}} \int_{\Lie } e^{-k \varphi (\xi, x)}   f(\exp( i\xi).x,z_1,z_2,k) \delta (\xi , x)   | dt_1 ...d t_\ell |(\xi)  
\end{gather*}
We estimate this integral by the stationary phase lemma. As we
already saw in the proof of proposition \ref{sec:estimation}, we have  
$$ \partial_{t^i} \varphi (\xi, x)  = - 2 \mu^{\xi^i} ( \exp(i\xi).x) $$ 
and the critical set of $\varphi$ is $\{ 0\} \times
P$. The second derivatives are 
\begin{gather} \label{eq:sd}
 \partial_{t^j} \partial_{t^i} \varphi (0, x)  = 2 g ( \xi_i^\#, \xi
_j^ \#) (x) .
\end{gather}
Since this matrix is non-degenerate, $f'$ has an asymptotic expansion in power of $k^{-1}$ with
coefficients in $\Ci ( P \times M_r^2)$. The first order term is given
by 
\begin{xalignat*}{2}
 f'(x,z_1,z_2,k) = &  f(x,z_1,z_2,k) \delta (0,x)  \operatorname{det}
 ^{-\frac{1}{2}} [ \partial_{t^j} \partial_{t^i} \varphi (0, x) ] +
 O(k^{-1} ) 
\end{xalignat*}
By (\ref{eq:sd}) and lemma \ref{sec:lemme_mesure}, it follows that 
\begin{xalignat*}{2}
 f'(x,z_1,z_2,k) 
= & f(x,z_1,z_2,k) 2^{- \frac{\ell}{2}}\operatorname{det}
^{\frac{1}{2}} [ g ( \xi_i^\#, \xi _j^ \#) ] (x) +O(k^{-1})
\end{xalignat*}
By (\ref{eq:17}), we obtain for any $(x,z)$ in the momentum Lagrangian
$\La$ that 
$$ f'(x, z,z , k)=  \overline{f_Q (x)} f_R (x) \;
\mu_{P_z}(x) \; t_r(z,z) + O(k^{-1})$$
 Consequently  
$$ f''(z,z,z,k) = \int _{P_z} \overline{f_Q (x)} f_R (x) \;
\mu_{P_z}(x) \; t_r(z,z) + O(k^{-1}) $$
To end the proof we have to compute the integral (\ref{eq:4}). We
recognize the integral corresponding to the composition of two
Toeplitz operators (cf. \cite{oim_bt}). So we obtain the Schwartz
kernel of a Toeplitz operator whose principal symbol is the function
$g$ such that  
$$ f''(z,z,z,k) = g (z) t_r(z,z) + O(k^{-1}) $$
which completes the proof. 

\subsection{Proof of theorem \ref{sec:sous-princ}}

Let us sketch the proof. Consider a Fourier integral
operator $(Q_k) \in \Fourier$ and two Toeplitz operators $(S_k)$ and 
$(T_k)$ of $M$ and $M_r$ respectively. Then the  Schwartz kernel of $S_k Q_k
$ is the image of the Schwartz kernel of $Q_k$ by the map 
 $$ S_k \otimes
\id : \Hilbert_{k} \otimes \overline{\Hilbert}_{r,k} \rightarrow \Hilbert_{k} \otimes \overline{\Hilbert}_{r,k} $$ 
Similarly the Schwartz kernel of $Q_k$ is sent into the Schwartz
kernel of $Q_kT_k$ by 
$$ \id \otimes \overline{ T_k^*} :  \Hilbert_{k} \otimes \overline{\Hilbert}_{r,k} \rightarrow \Hilbert_{k} \otimes \overline{\Hilbert}_{r,k} $$ 
$(S_k \otimes \id)$ and $( \id \otimes \overline{T_k^*})$ are
Toeplitz operators of $M \times \overline{M}_r$. The Schwartz kernel
of $(Q_k)$ is a Lagrangian section of $M \times \overline{M}_r$
associated to the moment Lagrangian in the same way the eigenstates in
section \ref{sec:etats-propres} are Lagrangian sections associated to the fiber of the
integrable system. So we are reduced to consider the action
of a Toeplitz operator on a Lagrangian section. This gives another
Lagrangian section and in the favorable cases the subprincipal terms
may be computed as it was explained in the paper \cite{oim_hf}, theorems
3.3 and 3.4. This will prove the theorem.  

Let us apply this program. We denote by $p$ and $p_r$  the projection from $M \times M_r$ onto the
first and second factor. Let $f_0$, $f_1$ and $g_0$, $g_1$ be the
principal and subprincipal symbols of $(S_k)$ and $(T_k)$.  
Then it is easily seen that the principal and
subprincipal symbols of $(S_k \otimes \id)$ (resp. $(\id 
\otimes \overline{T^*_k})$) are $p^*f_0$ and $p^*f_1$  (resp. $p_r^*
g_0 $ and $p_r^* g_1$).
 
We denote by $(\tilde{Q}_k)$ the sequence of the Schwartz kernels
of $(Q_k)$. Let  $h \in \Ci(P)$ be the symbol of the Fourier integral
operator $(Q_k)$. Then the symbol of the Lagrangian section
$(\tilde{Q}_k)$ is the section $\si$ of $\delta \boxtimes \overline{\delta}_r
\rightarrow \La$ given by 
$$ (x,z) \in \La  \rightarrow \si ( x, z) = h(x)  t_{\delta} (x,z)
$$
where $t_\delta$ has been introduced in the beginning of section
\ref{sec:four-integr}. 

Consider a Toeplitz operator $(O_k)$ of $M \times
\overline{M_r}$ with principal symbol $e_0$. Then by theorem 3.3 of
\cite{oim_hf} the symbol of $(O_k \tilde{Q}_k)$ is the product of $\si$ by the
restriction of $e_0$ to the moment Lagrangian $\La$. Applying this with the
operator $O_k = S_k \otimes
\overline{T^*_k}$, we deduce the
first part of theorem \ref{sec:sous-princ}. Indeed, since 
$$S_k \otimes
\overline{T^*_k} = ( S_k \otimes \id ) \circ (\id \otimes \overline{T^*_k}),$$ the principal
symbol of $(O_k)$ is $(p^* f_0) (p_r^*
g_0)$. Identifying $\Ci ( \La)$ with $\Ci (P)$, the restriction of
this symbol to the moment Lagrangian becomes $(j^* f_0 )( \pi^*
f_1)$, where $j$ is the injection $P \rightarrow M$ and $\pi$ is the
projection $P \rightarrow M_r$.  So the symbol of the Fourier integral
operator $(S_k Q_k T_k)$ is $(j^* f_0 )( \pi^*
f_1) h$. This proves the first assertion in theorem
\ref{sec:sous-princ}. 

The proof of the second one is the difficult part. Assume that the restriction of the principal symbol $e_0$ to the moment
Lagrangian vanishes. Then the symbol of $(O_k.\tilde{Q}_k)$ vanishes and
$(k O_k \tilde{Q}_k)$ is another Lagrangian section. By theorem 3.4 of
\cite{oim_hf}, its symbol is equal to  
\begin{gather} \label{eq:ss}
 \bigl( i^ * e_1 + \frac{1}{i} \dLie_X^{\delta \boxtimes \overline{\delta_r}} \bigr) \si
\end{gather}
Here $e_1$ is the subprincipal symbol of $(O_k)$ and $i^*e_1$ its
restriction to the momentum Lagrangian $\La$.  Since $e_0$
is constant over the Lagrangian manifold $\La$, its
Hamiltonian vector field is tangent to $\La$. $X$ is defined as the restriction 
to $\La$ of this Hamiltonian vector field.   

Finally the Lie
derivative $\dLie_X^{\delta \boxtimes \overline{\delta_r}}$ has the following sense. On the one hand, 
$\delta
\boxtimes \overline {\delta_r}$ is naturally a half-form bundle of $M
\times \overline{M_r}$. 
Recall that we denoted by $\varphi$ the isomorphism $\delta^2
\rightarrow  \dcan T^*M$ and by $\varphi_r $ the isomorphism
$\delta_r^2 \rightarrow \dcan T^*M_r$. Then the map
$$    \varphi ^2 \boxtimes \overline {\varphi_r}^2 : \delta ^{2}
\boxtimes \overline {\delta_r}^{2} \rightarrow  \dcan T^*M \boxtimes
\overline{\dcan T^*M_r} \simeq \dcan T^*(M \times \overline{M_r} )  $$
is an isomorphism of Hermitian holomorphic bundles. 
On the other hand the moment Lagrangian $\La$ being a Lagrangian submanifold
of $M
\times M_r^-$, the pull-back by the embedding $i : \La
\rightarrow M \times M_r$ induces an
isomorphism 
$$ i ^* : i ^ *( \dcan T^*(M \times \overline{M_r} )) \rightarrow \dvol T^*
\La \otimes \C$$
Let $\varphi_{\La}$ be the composition of these isomorphisms   
$$ \varphi_{\La} : i ^*(\delta  \boxtimes \overline {\delta_r})^2
\rightarrow \dvol T^*
\La \otimes \C $$
The Lie derivative $ \dLie_X^{\delta \boxtimes \overline{\delta_r}} $
is then defined as the first order differential operator such that for
any section $s$ of $i ^*(\delta \boxtimes \overline{\delta_r}) \rightarrow
\La$, one has
$$ \dLie_X \varphi_{\La} (s^2) = 2 \varphi_{\La} ( s \otimes
\dLie_X^{\delta \boxtimes \overline{\delta_r}}  s ) $$
Actually we already considered such a Lie derivative in remark
\ref{rem:symbol}. 

Let us apply this to the operator $O_k = S_k \otimes \id - \id \otimes
\overline{T^*_k}$.  Its principal symbol is $e_0 = p^* f_0 - p_r^* g_0
$. Assume that $f_0$ is $G$-invariant and that
$\pi ^* g_0 = j ^* f_0$. Then the restriction of $e_0$ to $\La$
vanishes. One has to compute the sum (\ref{eq:ss}). First,
$$ e_1 (x,z) = f_1 (z) - g_1 (x), \qquad   \forall (x, z) \in \La .$$
Second, identifying the momentum Lagrangian $\La$ with the zero level
set $P$, the
restriction $X$ of the Hamiltonian vector field of $e_0$ to $\La$ becomes
the restriction $Y$ of the
Hamiltonian vector field of 
$f_0$ to $P$. Then we will prove that 
\begin{gather} \label{eq:dl}
 (\dLie_X ^{\delta \boxtimes \overline{\delta}} \si)(x,z)  =
(\dLie_Y h ) (x)  . t_{\delta}(x), \qquad \forall (x, z) \in \La
\end{gather}
and theorem \ref{sec:sous-princ} follows directly. To show (\ref{eq:dl}), we compute $\varphi_\La ( t_{\delta}^2 )$.

\begin{lemma}
By identifying the momentum Lagrangian $\La$ with $P$,  $ \varphi_\La (
t^2_{\delta} )$ is the volume element of $P$ such that  
$$ \iota ({\gamma_\R}) \varphi_\La (
t^2_{\delta} ) = i^{n_r (n_r-2)} \pi ^* \om_r ^{\wedge n_r}$$
where  $n_r $ is the complex dimension of $M_r$, $\gamma_{\R}$ is the multivector $\xi_1^\# \wedge \ldots \wedge
\xi_\ell ^\#$ with $(\xi_i)$ an orthogonal base of $\Lie$.
\end{lemma}
From this we deduce that the Lie derivative with respect to $X$ of  $ \varphi_\La (
t^2_{\delta} )$ vanishes and consequently the same holds for the Lie
derivative of $ t_{\delta}$,
which proves (\ref{eq:dl}). 
\begin{proof}
By definition of $t_\delta$, if $x$ is a point of $P$ and $u \in
\delta_x$ is such that the norm of $[u]$ is equal to 1, then  
$$  t_{\delta} (x, \pi (x)) = u \otimes [\overline{ u}] $$
By identifying the momentum Lagrangian with $P$, one has 
$$ \varphi_\La (
t^2_{\delta} ) = j^ * \varphi (u ^2) \wedge \pi ^ * \overline{\varphi_r ([u]^2)}
$$
consequently,
$$ \iota ({\gamma_\R}) \varphi_\La (
t^2_{\delta} ) = j^ *  (\iota ({\gamma_\R}) \varphi (u ^2) )\wedge \pi
^ * \overline {\varphi_r ([u]^2)}
$$
Since $\varphi (u ^2) \in \dcan T^*M$, we have
$$ \iota ({\gamma_\R}) \varphi (u ^2) = \iota ({\gamma}) \varphi (u ^2) $$
where $\gamma$ is the multivector defined in (\ref{eq:7}). Next by
definition (\ref{eq:8}) of $\varphi_r$, we obtain that  
$$ \iota ({\gamma_\R}) \varphi (u ^2) = \pi_s^* \varphi_r ( [ u ] ^2) 
$$
Since $\pi = \pi_s \circ j$, it follows that
$$ j^* (  \iota ({\gamma_\R}) \varphi (u ^2) ) = \pi ^* \varphi_r ( [
u ] ^2)$$
Hence, 
$$ \iota ({\gamma_\R}) \varphi_\La (
t^2_{\delta} ) = \pi ^* \bigl( \varphi_r ( [ u ] ^2) \wedge \overline{\varphi_r
([u]^2 )} \bigr)
$$
Because $[u]$ is unitary, $\varphi_r ( [ u ] ^2 )$ is also of norm $1$,
we conclude.
\end{proof}

\bibliography{biblio}
\end{document}